\documentclass[12pt,a4paper]{article}
\usepackage{amssymb,amsmath}
\usepackage{hyperref}
\usepackage{amsthm}
\usepackage{ulem}
\newtheorem*{theorem}{Theorem}
\usepackage{chngcntr}
\usepackage{color}
\usepackage{graphicx}
\usepackage[dvipsnames]{xcolor}
\usepackage[symbol]{footmisc}
\usepackage[english,ngerman]{babel}
\usepackage[top=2.5cm,bottom=2.5cm,left=2.5cm,right=2.5cm]{geometry}

\begin{document}
%\color{red}
\begin{center}
\LARGE{\textbf{A fixed-point iteration method for the
\\number Pi with arbitrary odd order of
\\convergence based on the sine function}}
\[\]
\[\]
\large
{Alois Schiessl}
\footnote[1]{{University of Regensburg}\,\;E-Mail: \texttt{aloisschiessl@posteo.de}}
\selectlanguage{ngerman}
%\centerline{ }
\end{center}
%\centerline{}
\selectlanguage{english}
\centerline{***********************}
\centerline{\it In celebration of Pi Day}
\centerline{\it 3 - 14 - 2026}
\centerline{***********************}
\begin{abstract}
In this paper, we present a fixed point method for high-precision computation of number $\pi$ based on the sine function. Let $P\in \mathbb{N}$. We define the function:
\[
S\left(x\right)
=x+\sum_{k=1}^{P}\left(\prod_{\ell=1}^{k-1}\frac {2\,\ell-1}{2\,\ell}\right)\frac{\sin\left(x\right)^{2\,k-1}}{2\,k-1}\,.
\]
For every initial value $x_0$ sufficiently close to $\pi$, the sequence
\[x_{n+1}=x_n+S\left(x_{n}\right)\;;\,n=0,1,\ldots\]
is converging to $\pi$ with order of convergence exactly $\left(2\,P+1\right)$. The computational tests we performed demonstrate the efficiency of the method.
\selectlanguage{ngerman}
\[\]
\[\textbf{Zusammenfassung}\]
In dieser Abhandlung stellen wir ein Fixpunktverfahren zur Berechnung der Kreiszahl $\pi$ auf Basis der sinus Funktion vor. Es sei $P\in \mathbb{N}$. Wir definieren die Funktion:
\[
S\left(x\right)
=x+\sum_{k=1}^{P}\left(\prod_{\ell=1}^{k-1}\frac {2\,\ell-1}{2\,\ell}\right)\frac{\sin\left(x\right)^{2\,k-1}}{2\,k-1}\;.
\]
Für jeden Startwert $x_0$ hinreichend nahe bei $\pi$ konvergiert die Folge
\[x_{n+1}=x_n+S\left(x_{n}\right)\;;\,n=0,1,\ldots\]
gegen $\pi$ mit Konvergenzordnung genau $\left(2\,P+1\right)$. Anhand von praktischen Berechnungen zeigen wir die Effizienz des Verfahrens.
\[\text{Deutsche Version ab Seite 19}\]
\end{abstract}
\section{Introduction and main result}
We start with the $\arcsin$ power series \cite{HJ}
\[
\arcsin \left(t\right)
=t+\frac{1}{2}\cdot\frac{t^3}{3}+
\frac{1\cdot 3}{2\cdot 4}\cdot\frac{t^5}{5}+\ldots
=\sum_{k=1}^{\infty }\left(\prod_{\ell=1}^{k-1}\frac {2\,\ell-1}{2\,\ell}\right)\frac{\,t^{\,2\,k-1}}{2\,k-1}\,.
\]
The power series is absolutely convergent for $\left|t\right|\le 1$. We only use the first $P$ terms: 
\[
t+\frac{1}{2}\cdot\frac{t^3}{3}+
\frac{1\cdot 3}{2\cdot 4}\cdot\frac{t^5}{5}+\ldots
+\left(\prod_{\ell=1}^{P-1}\frac {2\,\ell-1}{2\,\ell}\right)\frac{t^{2\,P-1}}{2\,P-1}
=\sum_{k=1}^{P}\left(\prod_{\ell=1}^{k-1}\frac {2\,\ell-1}{2\,\ell}\right)\frac{t^{2\,k-1}}{2\,k-1}\,.
\]
Now substituting $t=\sin\left(x\right)$ leads to the finite sum
\[
\sum_{k=1}^{P}\left(\prod_{\ell=1}^{k-1}\frac {2\,\ell-1}{2\,\ell}\right)\frac{\sin\left(x\right)^{2\,k-1}}{2\,k-1}\,.
\]
With this sum we define a fixed point function that enables the calculation of $\pi$ with arbitrary odd order of convergence.
$ \\ $
$ \\ $
Now let's formulate our statement:
\begin{theorem}
$ \\ $
Let $P\in{\mathbb{N}}$. We define the function
\[S: \mathbb{R} \to \mathbb{R}\;;\]
\[x\mapsto S\left(x\right)\;;\]
\begin{align*}
S\left(x\right)
=x+\sum_{k=1}^{P}\left(\prod_{\ell=1}^{k-1}\frac{2\,\ell-1}{2\,\ell}\right)\frac{\sin\left(x\right)^{2\,k-1}}{2\,k-1}\,.
\end{align*}
Then, the following statements hold:
\begin{itemize}
\item[(a)]
$S\left(x\right)$ is a fixed point function with $\pi$ as fixed point, that means $S\left(\pi\right)=\pi\,.$
\item[(b)]
For the derivatives at the fixed point $\pi$, the following applies:
\begin{align*}
S^{\,\left(k\right)}\left(\pi\right)=\left\{
\begin{aligned}
0\qquad\qquad\qquad\;\; 1\leq &k\le 2\,P
\\
\\
\left(\prod_{\ell=1}^{P}\left(2\,\ell-1\right)\right)^{2}
\qquad\quad k=&2\,P+1
\\ 
\end{aligned} 
\right.
\end{align*}
\item[(c)] For every initial value $x_0$ sufficiently close to $\pi$, the sequence
\begin{align*}
x_{n+1}&=S\left(x_{n}\right)\,,\;n=0,1,\,\ldots
\end{align*}
is converging to $\pi$ with order of convergence exactly $\left(2\,P+1\right)\,.$
\end{itemize}
\end{theorem}
\section{Proof of the Theorem}
In this section, we prove the statements of the theorem.
\subsection{Proof of fixed point}
We only need plugging in $x=\pi$ into the fixed point function and obtain immediately
\[
S\left(\pi\right)=\pi+\sum _{k=1}^{P}\left(\prod_{\ell=1}^{k-1}\frac {2\,\ell-1}{2\,\ell}\right)\underbrace{\frac{\sin\left(\pi\right)^{2\,k-1}}{2\,k-1}}_{=0}=\pi\,.
\]
This proves the fixed point property.
\subsection{Proof of the derivatives}
Next, we turn to the derivatives. This requires a little bit more work.
\subsubsection{First Derivative}
Let's start with the first derivative. Differentiating the fixed point function yields
\begin{align*}
S\,'\left(x\right)
&=\left(x+\sum _{k=1}^{P}\left(\prod_{\ell=1}^{k-1}\frac {2\,\ell-1}{2\,\ell}\right)\frac{\sin\left(x\right)^{2\,k-1}}{2\,k-1}\right)'
\\
&=1+\sum _{k=1}^{P}\left(\prod_{\ell=1}^{k-1}\frac {2\,\ell-1}{2\,\ell}\right)\,\sin\left(x\right)^{2\,k-2}
\cdot\cos\left(x\right)\,.
\end{align*}
We then calculate the first derivative at the fixed point. To do this, we split the sum:
\begin{align*}
S\,'\left(x\right)=1
&+\sum _{k=1}^{1}\left(\prod_{\ell=1}^{k-1}\frac {2\,l-1}{2\,l}\right)\,\sin\left(x\right)^{2\,k-2}
\cdot\cos\left(x\right)
\\
&+\sum _{k=2}^{P}\left(\prod_{\ell=1}^{k-1}\frac {2\,l-1}{2\,l}\right)\,\sin\left(x\right)^{2\,k-2}
\cdot\cos\left(x\right)
\\
=1&+\cos\left(x\right)+\sum _{k=2}^{P}\left(\prod_{\ell=1}^{k-1}\frac {2\,l-1}{2\,l}\right)\,\sin\left(x\right)^{2\,k-2}
\cdot\cos\left(x\right)\,.
\end{align*}
Plugging in $x=\pi$ gives
\begin{align*}
S\,'\left(\pi\right)
=1&+\underbrace{\cos\left(\pi\right)}_{=-1}+\sum _{k=2}^{P}\left(\prod_{\ell=1}^{k-1}\frac {2\,l-1}{2\,l}\right)\,\underbrace{\sin\left(\pi\right)^{2\,k-2}}_{=0}\cdot\cos\left(x\right)
=1-1+0=0\,.
\end{align*}
One finally finds: $S\,'\left(\pi\right)=0\,.$ We need this result later.
\subsubsection{Second Derivative}
We differentiate the first derivative again and obtain
\begin{align*}
S\,''\left(x\right)
=&\sum _{k=1}^{P}\left(2\,k-2\right)\left(\prod_{\ell=1}^{k-1}\frac {2\,\ell-1}{2\,\ell}\right)\,\sin\left(x\right)^{2\,k-3}
\cdot\cos\left(x\right)^2
\\
+&\sum_{k=1}^{P}\left(\prod_{\ell=1}^{k-1}\frac {2\,\ell-1}{2\,\ell}\right)\,\sin\left(x\right)^{2\,k-2}
\cdot\left(-\sin\left(x\right)\right)
\\
=&\sum _{k=1}^{P}\left(\prod_{\ell=1}^{k-1}\frac{2\,\ell-1}{2\,\ell}\right)\left(2\,k-2\right)\,\sin\left(x\right)^{2\,k-3}
\cdot\cos\left(x\right)^2
\\
-&\sum_{k=1}^{P}\left(\prod_{\ell=1}^{k-1}\frac {2\,\ell-1}{2\,\ell}\right)\,\sin\left(x\right)^{2\,k-1}\,.
\end{align*}
Substituting $\cos\left(x\right)^2=1-\sin\left(x\right)^2$ gives
\begin{align*}
S\,''\left(x\right)
=&\sum _{k=1}^{P}\left(\prod_{\ell=1}^{k-1}\frac{2\,\ell-1}{2\,\ell}\right)\left(2\,k-2\right)\,\sin\left(x\right)^{2\,k-3}
\cdot\left(1-\sin\left(x\right)^2\right)
\\
-&\sum _{k=1}^{P}\left(\prod_{\ell=1}^{k-1}\frac{2\,\ell-1}{2\,\ell}\right)\,\sin\left(x\right)^{2\,k-1}
\\
=&\sum _{k=1}^{P}\left(\prod_{\ell=1}^{k-1}\frac {2\,\ell-1}{2\,\ell}\right)
\left(2\,k-2\right)\,\sin\left(x\right)^{2\,k-3}
\\
-&\sum _{k=1}^{P}\left(\prod_{\ell=1}^{k-1}\frac {2\,\ell-1}{2\,\ell}\right)\left(2\,k-2\right)\,\sin\left(x\right)^{2\,k-1}
\\
-&\sum _{k=1}^{P}\left(\prod_{\ell=1}^{k-1}\frac {2\,\ell-1}{2\,\ell}\right)\,\sin\left(x\right)^{2\,k-1}\,.
\end{align*}
Summarising the sums finally results in
\begin{align*}
S\,''\left(x\right)
&=\sum _{k=1}^{P}\left(\prod_{\ell=1}^{k-1}\frac {2\,\ell-1}{2\,\ell}\right)\left(2\,k-2\right)\,\sin\left(x\right)^{2\,k-3}
\\
&-\sum _{k=1}^{P}\left(\prod_{\ell=1}^{k-1}\frac {2\,\ell-1}{2\,\ell}\right)\left(2\,k-1\right)\,\sin\left(x\right)^{2\,k-1}\,.
\end{align*}
To make the matter more transparent, we will refer to the first sum as
\[
s_1=\sum _{k=1}^{P}\left(\prod_{\ell=1}^{k-1}\frac {2\,\ell-1}{2\,\ell}\right)\left(2\,k-2\right)\,\sin\left(x\right)^{2\,k-3}
\]
and the second sum as
\[
s_2=\sum _{k=1}^{P}\left(\prod_{\ell=1}^{k-1}\frac {2\,\ell-1}{2\,\ell}\right)\left(2\,k-1\right)\,\sin\left(x\right)^{2\,k-1}\,.
\]
The second derivative can then be written simply as
\[
S\,''\left(x\right)=s_1-s_2\,.
\]
Next step is splitting the sum $s_1$ as follows:
\begin{align*}
s_1&=\sum _{k=1}^{1}\left(\prod_{\ell=1}^{k-1}\frac {2\,\ell-1}{2\,\ell}\right)\left(2\,k-2\right)\,\sin\left(x\right)^{2\,k-3}
\\
&+\sum _{k=2}^{P}\left(2\,k-2\right)\left(\prod_{\ell=1}^{k-1}\frac {2\,\ell-1}{2\,\ell}\right)\,\sin\left(x\right)^{2\,k-3}\,.
\end{align*}
The first term is vanishing and it rests
\[
s_1=\sum _{k=2}^{P}\left(\prod_{\ell=1}^{k-1}\frac {2\,\ell-1}{2\,\ell}\right)\left(2\,k-2\right)\,\sin\left(x\right)^{2\,k-3}\,.
\]
Now shifting the summation index leads to
\[
s_1=\sum _{k=1}^{P-1}\left(\prod_{\ell=1}^{k}\frac {2\,\ell-1}{2\,\ell}\right)\left(2\,k\right)\,\sin\left(x\right)^{2\,k-1}\,.
\]
Expanding and rearranging the product term yields:
\begin{align*}
s_1
&=\sum _{k=1}^{P-1}
\left(\prod_{\ell=1}^{k-1}\frac {2\,\ell-1}{2\,\ell}\right)
\left(\prod_{l=k}^{k}\frac {2\,\ell-1}{2\,\ell}\right)
\left(2\,k\right)\,\sin\left(x\right)^{2\,k-1}
\\
&=\sum _{k=1}^{P-1}
\left(\prod_{\ell=1}^{k-1}\frac {2\,\ell-1}{2\,\ell}\right)
\left(\frac {2\,k-1}{2\,k}\right)\left(2\,k\right)\,\sin\left(x\right)^{2\,k-1}
\\
&=\sum _{k=1}^{P-1}
\left(\prod_{\ell=1}^{k-1}\frac {2\,\ell-1}{2\,\ell}\right)
\left(2\,k-1\right)\,\sin\left(x\right)^{2\,k-1}\,.
\end{align*}
Next, let's look at $s_2$. Also splitting the sum and simplifying gives a more simplified form:
\begin{align*}
s_2
&=\sum _{k=1}^{P}\left(\prod_{\ell=1}^{k-1}\frac {2\,\ell-1}{2\,\ell}\right)\left(2\,k-1\right)\,\sin\left(x\right)^{2\,k-1}
\\
&=\sum _{k=1}^{P-1}\left(\prod_{\ell=1}^{k-1}\frac {2\,\ell-1}{2\,\ell}\right)\left(2\,k-1\right)\,\sin\left(x\right)^{2\,k-1}
\\
&+\sum _{k=P}^{P}\left(\prod_{\ell=1}^{k-1}\frac {2\,\ell-1}{2\,\ell}\right)\left(2\,k-1\right)\,\sin\left(x\right)^{2\,k-1}
\\
&=\sum _{k=1}^{P-1}\left(\prod_{\ell=1}^{k-1}\frac {2\,\ell-1}{2\,\ell}\right)\left(2\,k-1\right)\,\sin\left(x\right)^{2\,k-1}
\\
&+\left(2\,P-1\right)\left(\prod_{\ell=1}^{P-1}\frac{2\,\ell-1}{2\,\ell}\right)\,\sin\left(x\right)^{2\,P-1}\,.
\end{align*}
Up to now we have
\begin{align*}
s_1&=\sum _{k=1}^{P-1}\left(\prod_{\ell=1}^{k-1}\frac {2\,\ell-1}{2\,\ell}\right)\left(2\,k-1\right)\,\sin\left(x\right)^{2\,k-1}\;;
\\
s_2&=\sum _{k=1}^{P-1}\left(\prod_{\ell=1}^{k-1}\frac {2\,\ell-1}{2\,\ell}\right)\left(2\,k-1\right)\,\sin\left(x\right)^{2\,k-1}
\\
&+\left(2\,P-1\right)\left(\prod_{\ell=1}{P-1}\frac {2\,\ell-1}{2\,\ell}\right)\,\sin\left(x\right)^{2\,P-1}\,.
\end{align*}
When used, we get
\begin{align*}
S\,''\left(x\right)
=&s_1-s_2
\\
=&\sum _{k=1}^{P-1}\left(\prod_{\ell=1}^{k-1}\frac {2\,\ell-1}{2\,\ell}\right)\left(2\,k-1\right)\,\sin\left(x\right)^{2\,k-1}
\\
-&\sum _{k=1}^{P-1}\left(\prod_{\ell=1}^{k-1}\frac {2\,\ell-1}{2\,\ell}\right)\left(2\,k-1\right)\,\sin\left(x\right)^{2\,k-1}
\\
-&\left(2\,P-1\right)\left(\prod_{\ell=1}^{P-1}\frac {2\,\ell-1}{2\,\ell}\right)\,\sin\left(x\right)^{2\,P-1}\,.
\end{align*}
$\\ $
The sums are vanishing and it remains the simple term:
\[
S\,''\left(x\right)=
-\left(2\,P-1\right)\left(\prod_{\ell=1}^{P-1}\frac{2\,\ell-1}{2\,\ell}\right)\,\sin\left(x\right)^{2\,P-1}\,.
\]
Plugging in the fixed point $x=\pi$ gives $S\,''\left(\pi\right)=\ldots=0\,.$
\subsubsection{Higher derivatives}
We still have to prove that the following statement holds at the fixed point $x=\pi:$
\begin{align*}
S^{\,\left(k\right)}\left(\pi\right)=\left\{
\begin{aligned}
0\qquad\qquad\qquad\;\; 1\leq &k\le 2\,P
\\
\\
\left(\prod_{\ell=1}^{P}\left(2\,\ell-1\right)\right)^{2}
\qquad\quad k=&2\,P+1
\\ 
\end{aligned} 
\right.
\end{align*}
We must differentiate $\left(2\,P+1\right)$ times $S\left(x\right)$ and then plugging in $x=\pi$. Using the second derivative $S\,''\left(x\right)$ the task simplifies to 
\begin{align*}
\left.\bigl(S\left(x\right)\bigr)^{\,\left(2\,P+1\right)}\right|_{x=\pi}
&=\left.\bigl(S\,''\left(x\right)\bigr)^{\,\left(2\,P-1\right)}\right|_{x=\pi}
\\
&=\left.\left(-\left(2\,P-1\right)\left(\prod\limits_{l=1}^{P-1}{\frac{2l-1}{2\,\ell}}\right)\sin {{\left(x\right)}^{2\,P-1}}\right)^{\left(2\,P-1\right)}\right|_{x=\pi}
\\
&=-\left(2\,P-1\right)\left(\prod\limits_{l=1}^{P-1}{\frac{2l-1}{2\,\ell}}\right)\left.\left(\sin {{\left(x\right)}^{2\,P-1}}\right)^{\left(2\,P-1\right)}\right|_{x=\pi}\,.
\end{align*}
On the right-hand side, we need at $x=\pi$ the $\left(2\,P-1\right)^{th}$ derivative of $\sin\left(x\right)^{2\,P-1}$. To do this, we proceed as follows: We expand $\sin\left(x\right)$ into a Taylor series of degree 1 at the fixed point $\pi$ with Lagrange remainder:
\[
\sin\left(x\right)
=\sin\left(\pi\right)
+\sin'\left(\pi\right)\left(x-\pi\right)
+\frac{\sin''\left(\xi\right)}{2}\left(x-\pi\right)^{2}
\]
where $\xi$ is between $x$ and $\pi$. Since $\sin\left(\pi\right)=0$ and $\sin'\left(\pi\right)=\cos\left(\pi\right)=-1$, we obtain
\[
\sin\left(x\right)
=-\left(x-\pi\right)+\frac{\sin''\left(\xi\right)}{2}\left(x-\pi\right)^{2}\,.
\]
Factoring out $\left(x-\pi\right)$ yields
\[
\sin\left(x\right)
=\left(x-\pi\right)\left(-1+\frac{\sin''\left(\xi\right)}{2}\left(x-\pi\right)\right)\,.
\]
Next we raise the quation to the $\left(2\,P-1\right)^{\,th}$ power:
\[
\sin\left(x\right)^{2\,P-1}
=\left(x-\pi\right)^{2\,P-1}
\left(-1+\frac{\sin{\,''}\left(\xi\right)}{2}\left(x-\pi\right)\right)^{2\,P-1}\,.
\]
We calculate the higher derivatives of $\sin\left(x\right)^{2\,P-1}$ by using the Leibniz rule \cite{LE}.
Leibniz rule states that if $u\left(x\right)$ and $v\left(x\right)$ are $n$-times differentiable functions, then the product $u\left(x\right)\cdot v\left(x\right)$ is also $n$-times differentiable and its $n^{\,th}$ derivative is given by the sum
\[
\bigl(u\left(x\right)\cdot v\left(x\right)\bigr)^{\left(n\right)}
=\sum\limits_{k=0}^{n}\binom{n}{k}\,u\left(x\right)^{\left(k\right)}\cdot v\left(x\right)^{\left(n-k\right)}\,.
\]
Here we want to understand $u\left(x\right)^{\left(0\right)}=u\left(x\right)$ and $v\left(x\right)^{\left(0\right)}=v\left(x\right)\,.$
$\\ $
$\\ $
In our case we have:
\[
n=2\,P-1\;;
\]
\[
u\left(x\right)=\left(x-\pi\right)^{2\,P-1}\;;
\]
\[
v\left(x\right)=\left(-1+\frac{\sin''\left(x\right)}{2}\left(x-\pi\right)\right)^{2\,P-1}\,.
\]
Thus we obtain
\[
\sin\left(x\right)^{2\,P-1}=
\underbrace{{\left(x-\pi\right)}^{2\,P-1}}_{u\left(x\right)}
\cdot
\underbrace{{{\left(-1+\frac{\sin^{\,''}{{\left(\xi\right)}}}{2}\left(x-\pi\right)\right)}^{2\,P-1}}}_{v\left(x\right)}\,.
\]
Now applying Leibniz rule gives
\begin{align*}
&\left(\sin\left(x\right)^{2\,P-1}\right)^{\left({2\,P-1}\right)}
\\
&=\sum\limits_{k=0}^{{2\,P-1}}\binom{{2\,P-1}}{k}
\,\left(\left(x-\pi\right)^{2\,P-1}\right)^{\left(k\right)}
\cdot
\left(\left(-1+\frac{\sin^{\,''}{{\left(\xi\right)}}}{2}\left(x-\pi\right)\right)^{{2\,P-1}}\right)^{\left({2\,P-k-1}\right)}\,.
\end{align*}
We first focus on the sequence of derivatives
\[
u\left(x\right)^{\left(k\right)}
=\left(\left(x-\pi\right)^{{2\,P-1}}\right)^{\left(k\right)},k=0,1,\ldots,{2\,P-1}\,.
\]
From a mathematical handbook \cite{SO}, we derive the following formula:
\begin{align*}
\Bigl(\left(x-\pi\right)^{2\,P-1}\Bigr)^{\left(k\right)}
=\frac{\left(2\,P-1\right)\,!}{\left(2\,P-1-k\right)\,!}\left(x-\pi\right)^{2\,P-1-k}\,.
\end{align*}
$\\ $
In particular, we note that for $k<2\,P-1$ in the formula
\[
u\left(x\right)^{\left(k\right)}=
\left(\left(x-\pi\right)^{{2\,P-1}}\right)^{\left(k\right)}
\]
the factor
$\left(x-\pi\right)$ always occurs and thus for $x=\pi$ the terms are vanishing. Only the case $k=\left(2\,P-1\right)$ leads to a non-zero result, namely the product
\[
{\left.u{{\left(x\right)}^{\left(2\,P-1\right)}}\right|}_{x=\pi}=\left(2\,P-1\right)\,!\,.
\]
From the sum
\[
\sum\limits_{k=0}^{{2\,P-1}}\binom{{2\,P-1}}{k}
\,\left(\left(x-\pi\right)^{2\,P-1}\right)^{\left(k\right)}
\cdot
\left(\left(-1+\frac{\sin^{\,''}{{\left(\xi\right)}}}{2}\left(x-\pi\right)\right)^{{2\,P-1}}\right)^{\left({2\,P-k-1}\right)}\,.
\]
only the term with $k=\left(2\,P-1\right)$ remains at the fixed point $x=\pi$. We can evaluate
\begin{align*}
&{\left.{{\left(\sin{{\left(x\right)}^{2\,P-1}}\right)}^{\left(2\,P-1\right)}}\right|}_{x=\pi}
\\
&=
\underbrace{\sum\limits_{k={2\,P-1}}^{{2\,P-1}}\binom{{2\,P-1}}{{2\,P-1}}\left.\left(\left(x-\pi\right)^{2\,P-1}\right)^{\left(2\,P-1\right)}\right|_{x=\pi}}_{=\left(2\,P-1\right)\,!}
\cdot
\underbrace{\left.\left(\left(-1
+\frac{\sin''{\left(\xi\right)}}{2}\left(x-\pi\right)\right)^{2\,P-1}\right)\right|_{x=\pi}}_{=-1}\,.
\end{align*}
We obtain the simple result
\[
{\left.{{\left(\sin{{\left(x\right)}^{2\,P-1}}\right)}^{\left(2\,P-1\right)}}\right|}_{x=\pi}
=-\left(2\,P-1\right)\,!\,.
\]
Our task was the calculation of
\begin{align*}
\left.S^{\,\left(2\,P+1\right)}\left(x\right)\right|_{x=\pi}
=\left.\left(-\left(2\,P-1\right)\left(\prod\limits_{l=1}^{P-1}{\frac{2l-1}{2\,\ell}}\right)\sin {{\left(x\right)}^{2\,P-1}}\right)^{\left(2\,P-1\right)}\right|_{x=\pi}\,.
\end{align*}
Now plugging in yields
\begin{align*}
S^{\,\left(2\,P+1\right)}\left(\pi\right)
&=-\left(2\,P-1\right)\left(\prod\limits_{\ell=1}^{P-1}{\frac{2\ell-1}{2\,\ell}}\right)
\bigl(-\left(2\,P-1\right)\,!\bigr)
\\
&=\left(2\,P-1\right)\cdot\left(2\,P-1\right)\,!\left(\prod\limits_{l=1}^{P-1}{\frac{2l-1}{2l}}\right)\,.
\end{align*}
We we can simplify the product term on the right hand side as follows:
\begin{align*}
\left(2\,P-1\right)\cdot\left(2\,P-1\right)\,!\cdot
\left(\prod\limits_{\ell=1}^{P-1}{\frac{2\ell-1}{2\,\ell}}\right)
&=\frac{\left(2\,P-1\right)\,!}{\prod\limits_{\ell=1}^{P-1}\left({2\,\ell}\right)}\cdot
\left(2\,P-1\right)\cdot\prod\limits_{\ell=1}^{P-1}\left({2\ell-1}\right)
\\
&=\frac{\left(2\,P-1\right)\,!}{\prod\limits_{\ell=1}^{P-1}\left({2\,\ell}\right)}\cdot\prod\limits_{\ell=1}^{P}\left({2\ell-1}\right)\,.
\end{align*}
Next we focus on the term
\[
\frac{\left(2\,P-1\right)\,!}{\prod\limits_{\ell=1}^{P-1}\left({2\,\ell}\right)}\,.
\]
Because it's a finite product, we can write down the factors directly
\[
\frac{\left(2\,P-1\right)\,!}{\prod\limits_{\ell=1}^{P-1}\left({2\,\ell}\right)}
=\frac
{1\cdot 2\cdot 3\ldots\left({2\,P-2}\right)\left({2\,P-1}\right)}
{2\cdot 4\cdot 6\ldots\left(2\,P-2\right)\qquad\qquad}\,.
\]
We cancel out the even factors and only the product of the odd numbers remains:
\[
\frac{\left(2\,P-1\right)\,!}{\prod\limits_{\ell=1}^{P-1}\left({2\,\ell}\right)}
=1\cdot 3\cdot 5\ldots\left({2\,P-3}\right)\left({2\,P-1}\right)
=\prod\limits_{\ell=1}^{P}\left({2\ell-1}\right)\,.
\]
Using this result we finally obtain
\begin{align*}
S^{\,\left(2\,P+1\right)}\left(\pi\right)
&=\frac{\left(2\,P-1\right)\,!}{\prod\limits_{\ell=1}^{P-1}\left({2\,\ell}\right)}\cdot\prod\limits_{\ell=1}^{P}\left({2\,\ell-1}\right)
=\left(\prod\limits_{\ell=1}^{P}\left({2\,\ell-1}\right)\right)
\cdot \left(\prod\limits_{\ell=1}^{P}\left({2\,\ell-1}\right)\right)
\\
&=\left(\prod\limits_{\ell=1}^{P}\left({2\,\ell-1}\right)\right)^2\,.
\end{align*}
Putting all the pieces together gives us the desired result
\begin{align*}
S^{\,\left(k\right)}\,{\left(\pi\right)}=\left\{
\begin{aligned}
0\qquad\qquad\qquad\;\; 1\le\,&k\leq 2\,P
\\
\\ 
\;\left(\prod\limits_{\ell=1}^{P}\left({2\ell-1}\right)\right)^2\,\;\qquad k = &2\,P+1 
\end{aligned} 
\right.
\end{align*}
This completes the proof of the derivations.
\subsection{Proof of convergence}
We do the prove by using Banach's fixed point theorem \cite {WG} in a special form for continuously differentiable functions in ${\mathbb{R}}^{1}\,.$
$ \\ $
$ \\ $
\underline{Banach fixed point theorem}
\\
Let $U\subseteq R$ be a closed subset; furthermore, let $F:U\rightarrow R$ be a mapping with the following properties\[
F\left(U\right)\subseteq U \,.
\]
There exists $0<L<1$ such that
\[
\left|F\,'\left(x\right)\right| \leq L\,,\; \text{for\,\;all} \;x\in U \,.
\]
Then the following statements hold:
\begin{itemize}
\item[(a)]
There is exactly one fixed point ${{x}^{*}}\in U$ of $F\,.$
\item[(b)]
For every initial value $x_0\in U$, the sequence
$x_{n+1}=F\left(x_{{n}}\right)$
converges to $x^*\,.$
\end{itemize}
$\\ $
\underline{Proof of convergence} %\cite {WG}
$\\ $
We have already proven that $S\left(\pi\right)=\pi$ and $S\,'\left(\pi\right)=0$ hold true. Then, for reasons of continuity, there exists a neighbourhood at the fixed point $\pi$
\[U_{\delta}=\left\{x\in\mathbb{R}:\left|x-\pi\right|\leq\delta\right\}\]
in which the following applies: $S\left(U_\delta\right)\subseteq U_\delta$ and $\left|S\,'\left(x\right)\right|<L<1$ for all $x\in U_{\delta}$.
Thus, the conditions for the fixed point theorem are satisfied and convergence follows.
\subsection{Order of convergence}
Finally, we have to proof that the order of convergence is exactly $\left(2\,P+1\right)$.
We expand $S\left(x\right)$ into a Taylor series around the fixed point $\pi:$
\[
S\left(x\right)=
S\left(\pi\right)
+S\,'\left(\pi\right)\left(x-\pi\right)
+\ldots
+S\,^{\,\left(2\,P\right)}\left(\pi\right)
\frac{\left(x-\pi\right)^{2\,P}}{\left(2\,P\right)\,!}
+S\,^{\,\left(2\,P+1\right)}\left(\xi\right)
\frac{\left(x-\pi\right)^{2\,P+1}}{\left(2\,P+1\right)\,!}\,,
\]
where $\xi$ is between $x$ and $\pi$. Let us remind $S\left(\pi\right)=\pi$ and the differentiation results
\begin{align*}
S^{\,\left(k\right)}\left(\pi\right)=\left\{
\begin{aligned}
0\qquad\qquad\qquad\;\; 1\leq &k\le 2\,P
\\
\\
\left(\prod_{\ell=1}^{P}\left(2\,\ell-1\right)\right)^{2}
\qquad\quad k=&2\,P+1
\\ 
\end{aligned} 
\right.
\end{align*}
The Taylor series thus simplifies to
\[
S\left(x\right)=
\underbrace{S\left(\pi\right)}_{=\pi}
+\underbrace{S\,'\left(\pi\right)\left(x-\pi\right)
+\ldots
+S\,^{\,\left(2\,P\right)}\left(\pi\right)
\frac{\left(x-\pi\right)^{2\,P}}{\left(2\,P\right)\,!}}_{=\,0}
+S\,^{\,\left(2\,P+1\right)}\left(\xi\right)
\frac{\left(x-\pi\right)^{2\,P+1}}{\left(2\,P+1\right)\,!}
\]
\[
=
\pi
+S\,^{\,\left(2\,P+1\right)}\left(\xi\right)
\frac{\left(x-\pi\right)^{2\,P+1}}{\left(2\,P+1\right)\,!}\,.
\]
Now plugging in $x=x_n$ sufficiently close to $\pi$, we obtain
\begin{align*}
S\left(x_{n}\right)=\pi+S\,^{\left(2\,P+1\right)}\left(\xi_{n}\right)
\frac{\left(x_{n}-\pi\right)^{2\,P+1}}{\left(2\,P+1\right)\,!}\,,
\end{align*}
where $\xi_n$ is between $x_n$ and $\pi$. Since $S\left(x_n\right)=x_{n+1}$, we get
\begin{align*}
x_{n+1}=\pi
+S\,^{\left(2\,P+1\right)}\left(\xi_{n}\right)
\frac{\left(x_{n}-\pi\right)^{2\,P+1}}{\left(2\,P+1\right)\,!}\,.
\end{align*}
Taking $\pi$ to the left side and dividing by $\left(x_{n+1}-\pi\right)^{2\,P+1}$ yields
\[
\frac{x_{n+1}-\pi}{\left(x_{n}-\pi\right)^{2\,P+1}}=
\frac{S\,^{\left(2\,P+1\right)}\left(\xi_{n}\right)}{\left(2\,P+1\right)\,!}\,.
\]
As $x_n\rightarrow \pi$, since $ \xi_n$ is trapped between $x_n$ and $\pi$ we conclude, by the continuity of $S^{\,\left(P+1\right)}\left(x\right)$ at $\pi$, that
\begin{align*}
\underset{n\to \infty}{\mathop{\lim }}\; \frac{x_{n+1}-\pi}{\left(x_{n}-\pi\right)^{2\,P+1}}
=\underset{n\to \infty}{\mathop{\lim }}\;
\frac{S\,^{\left(2\,P+1\right)}\left(\xi_{n}\right)}{\left(2\,P+1\right)\,!}
=
\frac{S\,^{\left(2\,P+1\right)}\left(\pi\right)}{\left(2\,P+1\right)\,!}\,.
\end{align*}
We use an earlier result
\[
S\,{^{\left({2\,P+1}\right)}\left(\pi\right)}
=\left(\prod\limits_{\ell=1}^{P}\left({2\,\ell-1}\right)\right)^2\,.
\]
Plugging in gives
\begin{align*}
\underset{n\to \infty}{\mathop{\lim }}\;\;
\frac{x_{n+1}-\pi}{\left(x_{n}-\pi\right)^{2\,P+1}}
=\frac{1}{\left(2\,P+1\right)\,!}\left(\prod\limits_{\ell=1}^{P}\left({2\,\ell-1}\right)\right)^2\,.
\end{align*}
This shows that convergence is exactly of the order $\left(2\,P+1\right)$, and
\[
\frac{1}{\left(2\,P+1\right)\,!}\left(\prod\limits_{\ell=1}^{P}\left({2\,\ell-1}\right)\right)^2
\]
is the asymptotic error constant.
\section{\texorpdfstring{Practical computation of\;$\pi$}{}}
Let $P\in\mathbb{N}$. That gives convergence order $\left(2\,P+1\right)$. We evaluate the fixed point function:
\[
S\left(x\right)=x+\sum _{k=1}^{P}\left(\left(\prod_{\ell=1}^{k-1}\frac {2\,\ell-1}{2\,\ell}\right)\frac{\sin\left(x\right)^{2\,k-1}}{2\,k-1}\right)\,.
\]
We observe that the somewhat complicated coefficients do not change throughout the calculation. It is therefore advantageous to calculate them beforehand and store them in a vector:
\[
c{\,_k}=\left(\prod_{\ell=1}^{k-1}\frac {2\,\ell-1}{2\,\ell}\right)\cdot\frac{1}{2\,k-1}\;;k=1\,,\ldots,\,{P}\,.
\]
The fixed point function is then simplified to
\[S\left(x\right)=\sum\limits_{k=1}^{P}c_{\,k}\,\sin\left(x\right)^{2\,k-1}.\]
From this, we derive the iteration rule
\[x_{n+1}=x_n+\sum\limits_{k=1}^{P}c_{\,k}\,\sin\left(x_n\right)^{2\,k-1}\;;\;n=0,1,\ldots\]
For practical computations, it is useful to shift the index. Instead of
\[x_{n+1}=x_n+\sum\limits_{k=1}^{P}c_{\,k}\,\sin\left(x_n\right)^{2\,k-1}\;;\;n=0,1,\ldots\]
we will write
\[x_{n}=x_{n-1}+\sum\limits_{k=1}^{P}c_{\,k}\,\sin\left(x_{n-1}\right)^{2\,k-1}\;;\;n=1,2,\ldots\,.\]
The index $n$ then runs synchronously with the number of iterations.
With an initial value $x_0$ sufficiently close to $\pi$, we start the iteration:
\[\underline{step\;\;1}\]
\[x_{1}=x_{0}+\sum\limits_{k=1}^{P}c_{\,k}\,\sin\left(x_{0}\right)^{2\,k-1}\,;\]

\[\underline{step\;\;2}\]
\[x_{2}=x_{1}+\sum\limits_{k=1}^{P}c_{\,k}\,\sin\left(x_{1}\right)^{2\,k-1}\,;\]

\[\underline{step\;\;3}\]
\[x_{3}=x_{2}+\sum\limits_{k=1}^{P}c_{\,k}\,\sin\left(x_{2}\right)^{2\,k-1}\,;\]

\[\vdots\]

\[\underline{step\;\;n}\]
\[x_{n}=x_{n-1}+\sum\limits_{k=1}^{P}c_{\,k}\,\sin\left(x_{n-1}\right)^{2\,k-1}\,;\]
$ \\ $
If, now $n\rightarrow \infty$, then $x_{n}\rightarrow \pi$ with order of convergence $\left(2\,P+1\right)$. With unlimited computing time and unlimited storage space, we can calculate $\pi$  with unlimited precision. Unfortunately, we do not have anything like that available. We must be content with a finite number of iterations and a finite value of $\pi$.
$ \\ $
$ \\ $
We require a termination criterion. For this purpose, we specify $\epsilon>0$ and perform the iteration until the following statement holds
\[\left|x_n-x_{n-1}\right|<\epsilon\,.\]
\subsection{Computational tests}
All figures and numerical calculations in this paper have been performed by using the computer algebra system \textbf{Maple 2025.2}.
$ \\ $
$ \\ $
In the first example let $P=1$. This gives order of convergence $\left(2\,P+1\right)=3$. Each additional step triples approximately the number of correct digits. We have only one coefficient $c_{\,1}=1\,.$ The fixed point function is very simple:
\[S\left(x\right)=x+\sin\left(x\right)\,.\]
%\newpage
Let's look at the functions $\sin\left(x\right)$ and $S\left(x\right)$ together with $y=x:$
\begin{center}
\includegraphics[width=1.0\linewidth]{"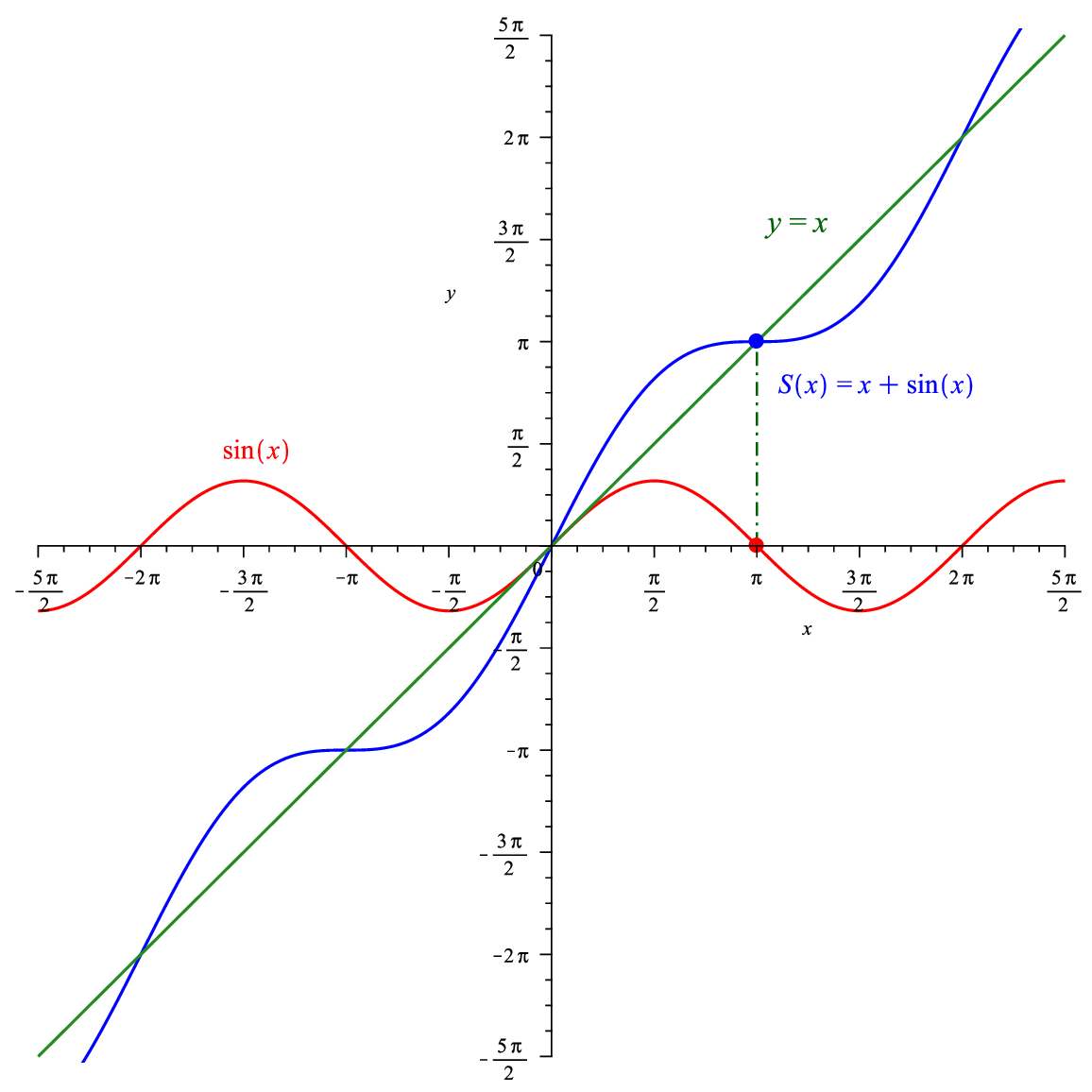"}
\end{center}
\textbf{\underline{figure 1\,:}}$\;\;\sin\left(x\right),\;S\left(x\right)=x+\sin\left(x\right)$ and $y=x$
$\\ $
$\\ $
From $S\left(x\right)$ we derive the iteration sequence \[x_{n}=x_{n-1}+\sin\left(x_{n-1}\right)\,.\]
We start with initial value $x_0=3$ and compute seven iterations
\[\underline{Step\;\;1}\]
\[x_{{1}}=x_{{0}}+\sin\left(x_{{0}}\right)\]
\[x_{{1}}=3.141120008059867222100744802808110279847\]
\[\left|x_{{1}}-x_{{0}}\right|=0.1411200080598672221007448028081102798469\]
\[\underline{Step\;\;2}\]
\[x_{{2}}=x_{{1}}+\sin\left(x_{{1}}\right)\]
\[x_{{2}}=3.141592653572195558734888568140879746743\]
\[\left|x_{{2}}-x_{{1}}\right|=0.0004726455123283366341437653327694668960596\]
\[\underline{Step\;\;3}\]
\[x_{{3}}=x_{{2}}+\sin\left(x_{{2}}\right)\]
\[x_{{3}}=3.141592653589793238462643383279501975927\]
\[\left|x_{{3}}-x_{{2}}\right|={1.759767972775481513862222918415964000597\times10^{-11}}\]

\[\underline{Step\;\;4}\]
\[x_{{4}}=x_{{3}}+\sin\left(x_{{3}}\right)\]
\[x_{{4}}=3.141592653589793238462643383279502884197\]
\[\left|x_{{4}}-x_{{3}}\right|={9.082700169421541422061207647042849447509\times10^{-34}}\]

\[\underline{Step\;\;5}\]
\[x_{{5}}=x_{{4}}+\sin\left(x_{{4}}\right)\]
\[x_{{5}}=3.141592653589793238462643383279502884197\]
\[\left|x_{{5}}-x_{{4}}\right|={1.248802280614662060876298822692655898920\times10^{-100}}\]

\[\underline{Step\;\;6}\]
\[x_{{6}}=x_{{5}}+\sin\left(x_{{5}}\right)\]
\[x_{{6}}=3.141592653589793238462643383279502884197\]
\[\left|x_{{6}}-x_{{5}}\right|={3.245860113595057511363635505332694019725\times10^{-301}}\]

\[\underline{Step\;\;7}\]
\[x_{{7}}=x_{{6}}+\sin\left(x_{{6}}\right)\]
\[x_{{7}}=3.141592653589793238462643383279502884197\]
\[x_{{7}}-x_{{6}}={5.699518230086813078688904262462295801452\times10^{-903}}\]
$\\ $
With seven iterations, we have received more than 900 digits of $\pi$. The negative exponents of the differences indicate a convergence order of $3$. The negative exponents approximately triple with each iteration step.
\subsection{\texorpdfstring{Computing one million digits of\;$\pi$\;with optimisations}{}}
The next example is more challenging. We want to compute one million digits of $\pi$. That means $\epsilon=10^{-1000000}$. In order to use the fixed point function efficiently for high-precision computation of $\pi$, a fast evaluation of $\sin\left(x\right)$ is required. There are many interesting publications on this topic \cite{RB},\cite{FJ1},\cite{FJ2},\cite{HG},\cite{RBA}. From the large number of proposed methods, we have selected the following three:
$\\ $
$\\ $
\underline{\bf{1.\;Selecting initial value very close to $\pi$}}
$\\ $
Now, we use as starting value
\[x_0=3.14159265358979324\]
which is already accurate to 18 digits.
%\[x_0=3.14159265358979323846264338327950288419716939937510582097494459230782\]
$ \\ $
$ \\ $
\underline{\bf{2.\;Increase in order of convergence}}
$\\ $
To ensure that the order of convergence takes full effect, we choose $P=4$. That gives order of convergence $2\cdot 4+1=9$. Then $x_n$ converges nonically to $\pi$, that is, each iteration approximately multiplies the number of correct digits by nine. Thus we need six iterations to obtain at least one million digits of $\pi$.
$ \\ $
$ \\ $
\underline{\bf{3.\;Increase in the number of Digits}}
$\\ $
The fixed point iteration
\[x_{n}=x_{n-1}+\sum\limits_{k=1}^{P}c_{\,k}\,\sin\left(x_{n-1}\right)^{2\,k-1}\]
is self-correcting. Each iteration must not be performed with the desired number of correct digits of $\pi$'s final result. We can start with lower precision.
$\\ $
$\\ $
The \textbf{Digits} environment variable controls the number of digits that Maple uses when making calculations with software floating-point numbers. In step 1 we then need an accuracy of $18\cdot 9=162$ \textbf{Digits}. In step 2 the needed accuracy is $18\cdot 9^2=1458$ \textbf{Digits} and so on until step 5 with accuracy of $18\cdot 9^5=1062882$ \textbf{Digits}. We have now reached the specified accuracy of at least one million digits. Therefore, in the final step 6, we no longer need to multiply the number of \textbf{Digits} by 9.
$ \\ $
$ \\ $
First determining the coefficients
\[c_1=1\;;\;c_2=\frac{1}{6}\;;\;c_3=\frac{3}{40}\;;\;c_4=\frac{5}{112}\]
 and the fixed point function:
\[
S\left(x\right)
=x+\sin\left(x\right)
+\frac{1}{6}\,\sin\left(x\right)^{3}
+\frac{3}{40}\,\sin\left(x\right)^{5}
+\frac{5}{112}\,\sin\left(x\right)^{7}\,.
\]
From $S\left(x\right)$ we derive the iteration rule:
\[x_{n}=x_{n-1}+\sin\left(x_{n-1}\right)
+\frac{1}{6}\sin\left(x_{n-1}\right)^{3}
+\frac{3}{40}\,\sin\left(x_{n-1}\right)^{5}
+\frac{5}{112}\,\sin\left(x_{n-1}\right)^{7}\;.\]
$\\ $
Before starting the computation we take a look at the involved functions:
\begin{center}
\includegraphics[width=1.0\linewidth]{"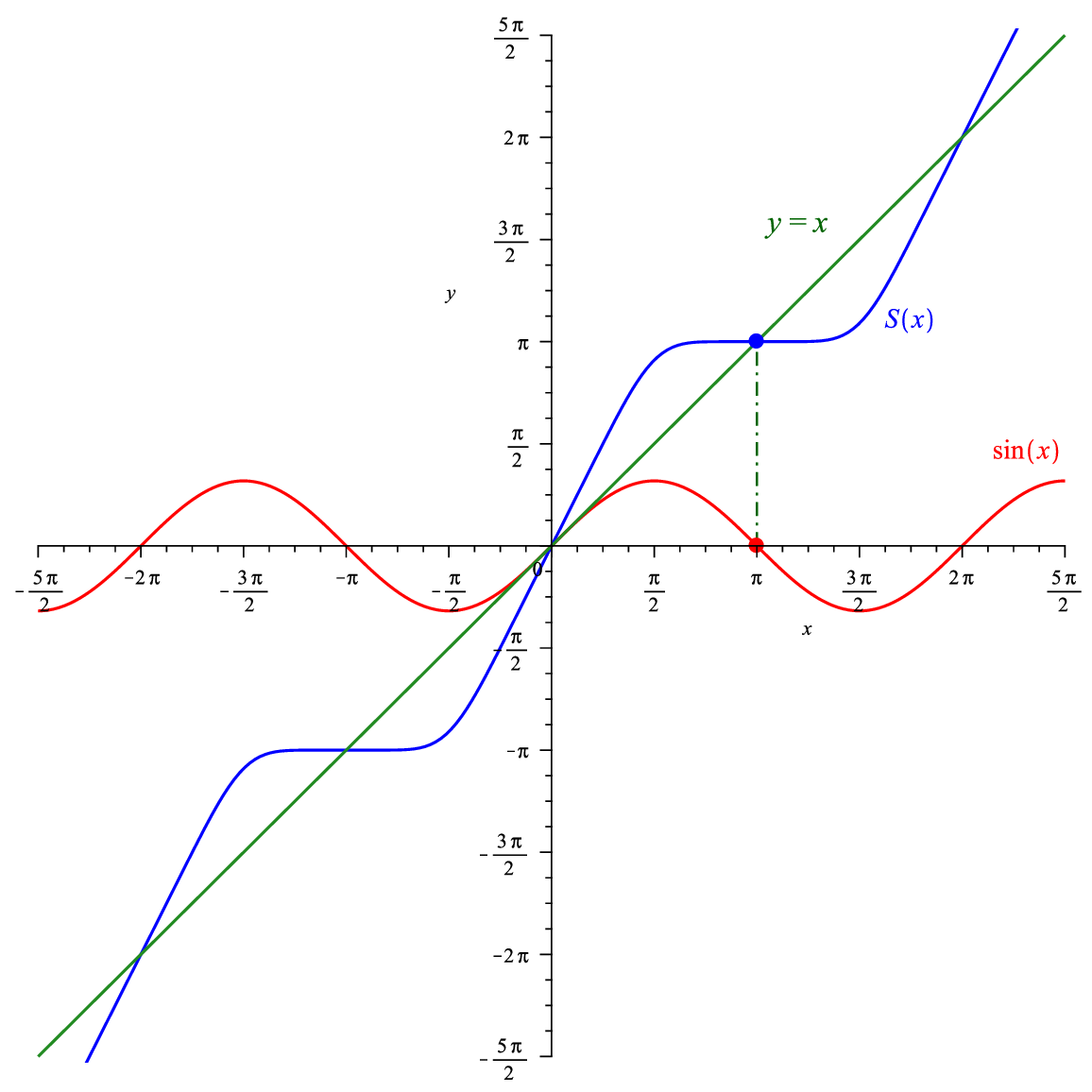"}
\end{center}
%\color{black}
%\color{Sepia}
\textbf{\underline{figure 2\,:}}$\;\;\sin\left(x\right),\;S\left(x\right)
=x+\sin\left(x\right)
+\frac{1}{6}\,\sin\left(x\right)^{3}
+\frac{3}{40}\,\sin\left(x\right)^{5}
+\frac{5}{112}\,\sin\left(x\right)^{7}$ and $y=x$
$\\ $
$\\ $
In the next $\pi$ calculation, we have implemented the three optimisation methods described above. We start with the initial value \[x_0=3.14159265358979324\] and compute six steps. In addition, we output the number of \textbf{Digits} with which Maple calculates at each step. We obtain:
\[\underline{step\;\;1}\]
\[Digits=18\cdot 9=162\]
\[x_{{1}}=x_{{0}}+\sin\left(x_{{0}}\right)+{\frac{\sin\left(x_{{0}}\right)^{3}}{6}}+{\frac{3\,\sin\left(x_{{0}}\right)^{5}}{40}}+{\frac{5\,\sin\left(x_{{0}}\right)^{7}}{112}}\]
\[x_{{1}}=3.141592653589793238462643383279502884197\]
\[\left|x_{{1}}-x_{{0}}\right|={1.537356616720497115802830600624894179025\times10^{-18}}\]

\[\underline{step\;\;2}\]
\[Digits=18\cdot 9^2=1458\]
\[x_{{2}}=x_{{1}}+\sin\left(x_{{1}}\right)+{\frac{\sin\left(x_{{1}}\right)^{3}}{6}}+{\frac{3\,\sin\left(x_{{1}}\right)^{5}}{40}}+{\frac{5\,\sin\left(x_{{1}}\right)^{7}}{112}}\]
\[x_{{2}}=3.141592653589793238462643383279502884197\]
\[\left|x_{{2}}-x_{{1}}\right|={5.897298061478894440355377051045069618036\times10^{-162}}\]

\[\underline{step\;\;3}\]
\[Digits=18\cdot 9^3=13122\]
\[x_{{3}}=x_{{2}}+\sin\left(x_{{2}}\right)+{\frac{\sin\left(x_{{2}}\right)^{3}}{6}}+{\frac{3\,\sin\left(x_{{2}}\right)^{5}}{40}}+{\frac{5\,\sin\left(x_{{2}}\right)^{7}}{112}}\]
\[x_{{3}}=3.141592653589793238462643383279502884197\]
\[\left|x_{{3}}-x_{{2}}\right|={2.621201614945054114130730043090727892025\times10^{-1453}}\]

\[\underline{step\;\;4}\]
\[Digits=18\cdot 9^4=118098\]
\[x_{{4}}=x_{{3}}+\sin\left(x_{{3}}\right)+{\frac{\sin\left(x_{{3}}\right)^{3}}{6}}+{\frac{3\,\sin\left(x_{{3}}\right)^{5}}{40}}+{\frac{5\,\sin\left(x_{{3}}\right)^{7}}{112}}\]
\[x_{{4}}=3.141592653589793238462643383279502884197\]
\[\left|x_{{4}}-x_{{3}}\right|={1.774677351983440263736354294020071931434\times10^{-13075}}\]

\[\underline{step\;\;5}\]
\[Digits=18\cdot 9^5=1062882\]
\[x_{{5}}=x_{{4}}+\sin\left(x_{{4}}\right)+{\frac{\sin\left(x_{{4}}\right)^{3}}{6}}+{\frac{3\,\sin\left(x_{{4}}\right)^{5}}{40}}+{\frac{5\,\sin\left(x_{{4}}\right)^{7}}{112}}\]
\[x_{{5}}=3.141592653589793238462643383279502884197\]
\[\left|x_{{5}}-x_{{4}}\right|={5.305059119433192986813252324813505545638\times10^{-117675}}\]

\[\underline{step\;\;6}\]
\[Digits=1062882\]
\[x_{{6}}=x_{{5}}+\sin\left(x_{{5}}\right)+{\frac{\sin\left(x_{{5}}\right)^{3}}{6}}+{\frac{3\,\sin\left(x_{{5}}\right)^{5}}{40}}+{\frac{5\,\sin\left(x_{{5}}\right)^{7}}{112}}\]
\[x_{{6}}=3.141592653589793238462643383279502884197\]
\[\left|x_{{6}}-x_{{5}}\right|={1.011178012461738293576866860651780969327\times10^{-1059070}}\]
$\\ $
With six steps, we obtained well over a million digits of $\pi$. The negative exponents of the differences show the convergence order $9$. The negative exponents approximately increase by a factor of 9 with each iteration step.
$\\ $
We performed all computational results by using a home-made PC with the following hardware configuration: Motherboard ASUS PRIME A320M-K with CPU AMD Ryzen 5 5600G 6 CORE 3.90-4.40 GHz and 32 GB RAM. The used software was MAPLE 2025.2 by  Maplesoft, Waterloo Maple Inc. This software was also used to assist with formatting the mathematical terms into LaTeX.

\selectlanguage{english}

\newpage
\selectlanguage{ngerman}
\begin{center}
\LARGE{\textbf{Fixpunkt-Verfahren zur Berechnung
\\der Kreiszahl Pi mit beliebiger ungerader
\\Konvergenzordnung auf Basis von sin(x)}}
\[\]
\[\]
\large
{Alois Schiessl}
%\vskip 0.25cm
%{E-Mail: aloisschiessl@web.de}
\footnote[1]{{University of Regensburg}\,\;E-Mail: \texttt{aloisschiessl@posteo.de}}
\selectlanguage{ngerman}
\centerline{ }
\end{center}
\centerline{}
\begin{center}
\end{center}
\centerline{***********************}
\centerline{\it Zur Feier des Pi Day}
\centerline{\it 3 - 14 - 2026}
\centerline{***********************}
\begin{abstract}
In dieser Abhandlung stellen wir ein Fixpunktverfahren zur Berechnung der Kreiszahl $\pi$ auf Basis des sinus vor. Es sei $P\in \mathbb{N}$. Wir definieren die Funktion:
\[
S\left(x\right)
=x+\sum_{k=1}^{P}\left(\prod_{\ell=1}^{k-1}\frac {2\,\ell-1}{2\,\ell}\right)\frac{\sin\left(x\right)^{2\,k-1}}{2\,k-1}\;.
\]
Für jeden Startwert $x_0$ hinreichend nahe bei $\pi$ konvergiert die Folge
\[x_{n+1}=x_n+S\left(x_{n}\right)\;;\,n=0,1,\ldots\]
gegen $\pi$ mit Konvergenzordnung genau $\left(2\,P+1\right)$. Anhand von praktischen Berechnungen zeigen wir die Effizienz des Verfahrens.
\end{abstract}
\setcounter{section}{0}
\section{Einführung und Hauptergebnis}
Wir beginnen mit der $\arcsin$ Potenzreihe \cite{HJ}
\[
\arcsin \left(t\right)
=t+\frac{1}{2}\cdot\frac{t^3}{3}+
\frac{1\cdot 3}{2\cdot 4}\cdot\frac{t^5}{5}+\ldots
=\sum_{k=1}^{\infty }\left(\prod_{\ell=1}^{k-1}\frac {2\,\ell-1}{2\,\ell}\right)\frac{\,t^{\,2\,k-1}}{2\,k-1}\,.
\]
Die Potenzreihe ist absolut konvergent für $\left|t\right|\le 1$. Wir verwenden nur die ersten $P$ Summanden: 
\[
t+\frac{1}{2}\cdot\frac{t^3}{3}+
\frac{1\cdot 3}{2\cdot 4}\cdot\frac{t^5}{5}+\ldots
+\left(\prod_{\ell=1}^{P-1}\frac {2\,\ell-1}{2\,\ell}\right)\frac{t^{2\,P-1}}{2\,P-1}
=\sum_{k=1}^{P}\left(\prod_{\ell=1}^{k-1}\frac {2\,\ell-1}{2\,\ell}\right)\frac{t^{2\,k-1}}{2\,k-1}\,.
\]
Jetzt substituieren wir $t=\sin\left(x\right)$ und erhalten die endliche Summe
\[
\sum_{k=1}^{P}\left(\prod_{\ell=1}^{k-1}\frac {2\,\ell-1}{2\,\ell}\right)\frac{\sin\left(x\right)^{2\,k-1}}{2\,k-1}\,.
\]
Diese Summe verwenden wir um eine Fixpunktfunktion zu definieren, die die Berechnung der Kreiszahl  $\pi$ mit jeder beliebigen ungeraden Konvergenzordnung ermöglicht.
$ \\ $
$ \\ $
Wir formulieren unser Theorem:
\begin{theorem}
$ \\ $
Es sei $P\in{\mathbb{N}}$. Wir definieren die Funktion
\[S: \mathbb{R} \to \mathbb{R}\;;\]
\[x\mapsto S\left(x\right)\;;\]
\begin{align*}
S\left(x\right)
=x+\sum_{k=1}^{P}\left(\prod_{\ell=1}^{k-1}\frac{2\,\ell-1}{2\,\ell}\right)\frac{\sin\left(x\right)^{2\,k-1}}{2\,k-1}\,.
\end{align*}
Es gelten die folgenden Aussagen:
\begin{itemize}
\item[(a)]
$S\left(x\right)$ ist eine Fixpunktfunktion mit $\pi$ als Fixpunkt. Das bedeutet  $S\left(\pi\right)=\pi\,.$
\item[(b)]
Für die Ableitungen im Fixpunkt $\pi$ gilt:
\begin{align*}
S^{\,\left(k\right)}\left(\pi\right)=\left\{
\begin{aligned}
0\qquad\qquad\qquad\;\; 1\leq &k\le 2\,P
\\
\\
\left(\prod_{\ell=1}^{P}\left(2\,\ell-1\right)\right)^{2}
\qquad\quad k=\,&2\,P+1
\\ 
\end{aligned} 
\right.
\end{align*}
\item[(c)] Es sei $x_0$ ein Startwert hinreichend nahe bei $\pi$. Dann konvergiert die Folge
\begin{align*}
x_{n+1}&=S\left(x_{n}\right)\,,\;n=0,1,\,\ldots
\end{align*}
gegen $\pi$ mit Konvergenzordnung genau $\left(2\,P+1\right)\,.$
\end{itemize}
\end{theorem}
\section{Beweis des Theorem}
In diesem Kapitel weisen wir die Aussagen zum Theorem nach. Wir benötigen dazu nur elementare Kenntnisse der Algebra und der klassischen Analysis.
\subsection{Fixpunktbeweis}
Am einfachsten ist die Fixpunktaussage $S\left(\pi\right)=\pi$ zu belegen. Dazu brauchen wir nur $x=\pi$ in die Fixpunktfunktion einzusetzen und erhalten
\[
S\left(\pi\right)=\pi+\sum _{k=1}^{P}\left(\prod_{\ell=1}^{k-1}\frac {2\,\ell-1}{2\,\ell}\right)\underbrace{\frac{\sin\left(\pi\right)^{2\,k-1}}{2\,k-1}}_{=0}=\pi\,.
\]
Damit ist die Fixpunkteigenschaft bewiesen.
\subsection{Beweis der Ableitungen}
Als nächstes wenden wir uns den Ableitungen zu. Das ist etwas mehr Aufwand nötig.
\subsubsection{Erste Ableitung}
Wir differenzieren die Fixpunktfunktion nach allen Regeln der Kunst und erhalten:
\begin{align*}
S\,'\left(x\right)
&=\left(x+\sum _{k=1}^{P}\left(\prod_{\ell=1}^{k-1}\frac {2\,\ell-1}{2\,\ell}\right)\frac{\sin\left(x\right)^{2\,k-1}}{2\,k-1}\right)'
\\
&=1+\sum _{k=1}^{P}\left(\prod_{\ell=1}^{k-1}\frac {2\,\ell-1}{2\,\ell}\right)\,\sin\left(x\right)^{2\,k-2}
\cdot\cos\left(x\right)\,.
\end{align*}
Als nächstes berechnen wir den Wert der ersten Ableitung im Fixpunkt $x=\pi$. Dazu rechnen wir den ersten Summanden aus und lassen den Index der restlichen Summe bei zwei beginnen:
\begin{align*}
S\,'\left(x\right)=1
&+\sum _{k=1}^{1}\left(\prod_{\ell=1}^{k-1}\frac {2\,l-1}{2\,l}\right)\,\sin\left(x\right)^{2\,k-2}
\cdot\cos\left(x\right)
\\
&+\sum _{k=2}^{P}\left(\prod_{\ell=1}^{k-1}\frac {2\,l-1}{2\,l}\right)\,\sin\left(x\right)^{2\,k-2}
\cdot\cos\left(x\right)
\\
=1&+\cos\left(x\right)+\sum _{k=2}^{P}\left(\prod_{\ell=1}^{k-1}\frac {2\,l-1}{2\,l}\right)\,\sin\left(x\right)^{2\,k-2}
\cdot\cos\left(x\right)\,.
\end{align*}
Setzen wir jetzt $x=\pi$ ein, so erhalten wir:
\begin{align*}
S\,'\left(\pi\right)
=1&+\underbrace{\cos\left(\pi\right)}_{=-1}+\sum _{k=2}^{P}\left(\prod_{\ell=1}^{k-1}\frac {2\,l-1}{2\,l}\right)\,\underbrace{\sin\left(\pi\right)^{2\,k-2}}_{=0}\cdot\cos\left(x\right)
=1-1+0=0\,.
\end{align*}
Schließlich ergibt sich: $S\,'\left(\pi\right)=0\,.$ Wir benötigen dieses Ergebnis später.
\subsubsection{Zweite Ableitung}
Wir differenzieren die erste Ableitung erneut. Das ergibt dann:
\begin{align*}
S\,''\left(x\right)
=&\sum _{k=1}^{P}\left(2\,k-2\right)\left(\prod_{\ell=1}^{k-1}\frac {2\,\ell-1}{2\,\ell}\right)\,\sin\left(x\right)^{2\,k-3}
\cdot\cos\left(x\right)^2
\\
+&\sum_{k=1}^{P}\left(\prod_{\ell=1}^{k-1}\frac {2\,\ell-1}{2\,\ell}\right)\,\sin\left(x\right)^{2\,k-2}
\cdot\left(-\sin\left(x\right)\right)
\\
=&\sum _{k=1}^{P}\left(\prod_{\ell=1}^{k-1}\frac{2\,\ell-1}{2\,\ell}\right)\left(2\,k-2\right)\,\sin\left(x\right)^{2\,k-3}
\cdot\cos\left(x\right)^2
\\
-&\sum_{k=1}^{P}\left(\prod_{\ell=1}^{k-1}\frac {2\,\ell-1}{2\,\ell}\right)\,\sin\left(x\right)^{2\,k-1}\,.
\end{align*}
Wir ersetzen $\cos\left(x\right)^2$ durch $\bigl(1-\sin\left(x\right)^2\bigr)$, lösen die Klammer auf und vereinfachen:
\begin{align*}
S\,''\left(x\right)
=&\sum _{k=1}^{P}\left(\prod_{\ell=1}^{k-1}\frac{2\,\ell-1}{2\,\ell}\right)\left(2\,k-2\right)\,\sin\left(x\right)^{2\,k-3}
\cdot\left(1-\sin\left(x\right)^2\right)
\\
-&\sum _{k=1}^{P}\left(\prod_{\ell=1}^{k-1}\frac{2\,\ell-1}{2\,\ell}\right)\,\sin\left(x\right)^{2\,k-1}
\\
=&\sum _{k=1}^{P}\left(\prod_{\ell=1}^{k-1}\frac {2\,\ell-1}{2\,\ell}\right)
\left(2\,k-2\right)\,\sin\left(x\right)^{2\,k-3}
\\
-&\sum _{k=1}^{P}\left(\prod_{\ell=1}^{k-1}\frac {2\,\ell-1}{2\,\ell}\right)\left(2\,k-2\right)\,\sin\left(x\right)^{2\,k-1}
\\
-&\sum _{k=1}^{P}\left(\prod_{\ell=1}^{k-1}\frac {2\,\ell-1}{2\,\ell}\right)\,\sin\left(x\right)^{2\,k-1}\,.
\end{align*}
Die letzten beiden Summen können wir zusammenfassen:
\begin{align*}
S\,''\left(x\right)
&=\sum _{k=1}^{P}\left(\prod_{\ell=1}^{k-1}\frac {2\,\ell-1}{2\,\ell}\right)\left(2\,k-2\right)\,\sin\left(x\right)^{2\,k-3}
\\
&-\sum _{k=1}^{P}\left(\prod_{\ell=1}^{k-1}\frac {2\,\ell-1}{2\,\ell}\right)\left(2\,k-1\right)\,\sin\left(x\right)^{2\,k-1}\,.
\end{align*}
Um die Rechnerei transparenter zu gestalten bezeichnen wir die erste Summe mit
\[
s_1=\sum _{k=1}^{P}\left(\prod_{\ell=1}^{k-1}\frac {2\,\ell-1}{2\,\ell}\right)\left(2\,k-2\right)\,\sin\left(x\right)^{2\,k-3}
\]
und die zweite Summe mit
\[
s_2=\sum _{k=1}^{P}\left(\prod_{\ell=1}^{k-1}\frac {2\,\ell-1}{2\,\ell}\right)\left(2\,k-1\right)\,\sin\left(x\right)^{2\,k-1}\,.
\]
Die zweite Ableitung schreibt sich dann ganz einfach als
\[
S\,''\left(x\right)=s_1-s_2\,.
\]
Wir zerlegen $s_1$, in dem wir den ersten Term ausreichen und den Index ab 2 laufen lassen:
\begin{align*}
s_1&=\sum _{k=1}^{1}\left(\prod_{\ell=1}^{k-1}\frac {2\,\ell-1}{2\,\ell}\right)\left(2\,k-2\right)\,\sin\left(x\right)^{2\,k-3}
\\
&+\sum _{k=2}^{P}\left(2\,k-2\right)\left(\prod_{\ell=1}^{k-1}\frac {2\,\ell-1}{2\,\ell}\right)\,\sin\left(x\right)^{2\,k-3}\,.
\end{align*}
Der erste Term verschwindet und es bleibt nur die Summe
\[
s_1=\sum _{k=2}^{P}\left(\prod_{\ell=1}^{k-1}\frac {2\,\ell-1}{2\,\ell}\right)\left(2\,k-2\right)\,\sin\left(x\right)^{2\,k-3}\,.
\]
Als nächstes nehmen wir eine Indexverschiebung vor:
\[
s_1=\sum _{k=1}^{P-1}\left(\prod_{\ell=1}{k}\frac {2\,\ell-1}{2\,\ell}\right)\left(2\,k\right)\,\sin\left(x\right)^{2\,k-1}\,.
\]
Der Produktterm in der Summe kann wie folgt vereinfacht werden:
\begin{align*}
s_1
&=\sum _{k=1}^{P-1}
\left(\prod_{\ell=1}^{k-1}\frac {2\,\ell-1}{2\,\ell}\right)
\left(\prod_{l=k}^{k}\frac {2\,\ell-1}{2\,\ell}\right)
\left(2\,k\right)\,\sin\left(x\right)^{2\,k-1}
\\
&=\sum _{k=1}^{P-1}
\left(\prod_{\ell=1}^{k-1}\frac {2\,\ell-1}{2\,\ell}\right)
\left(\frac {2\,k-1}{2\,k}\right)\left(2\,k\right)\,\sin\left(x\right)^{2\,k-1}
\\
&=\sum _{k=1}^{P-1}
\left(\prod_{\ell=1}^{k-1}\frac {2\,\ell-1}{2\,\ell}\right)
\left(2\,k-1\right)\,\sin\left(x\right)^{2\,k-1}\,.
\end{align*}
Ähnlich verfahren wir mit $s_2$. Diesmal lassen wir die ersten $\left(P-1\right)$ Summanden stehen und rechnen den letzten Summanden aus: 
\begin{align*}
s_2
&=\sum _{k=1}^{P}\left(\prod_{\ell=1}^{k-1}\frac {2\,\ell-1}{2\,\ell}\right)\left(2\,k-1\right)\,\sin\left(x\right)^{2\,k-1}
\\
&=\sum _{k=1}^{P-1}\left(\prod_{\ell=1}^{k-1}\frac {2\,\ell-1}{2\,\ell}\right)\left(2\,k-1\right)\,\sin\left(x\right)^{2\,k-1}
\\
&+\sum _{k=P}^{P}\left(\prod_{\ell=1}^{k-1}\frac {2\,\ell-1}{2\,\ell}\right)\left(2\,k-1\right)\,\sin\left(x\right)^{2\,k-1}
\\
&=\sum _{k=1}^{P-1}\left(\prod_{\ell=1}^{k-1}\frac {2\,\ell-1}{2\,\ell}\right)\left(2\,k-1\right)\,\sin\left(x\right)^{2\,k-1}
\\
&+\left(2\,P-1\right)\left(\prod_{\ell=1}^{P-1}\frac {2\,\ell-1}{2\,\ell}\right)\,\sin\left(x\right)^{2\,P-1}\,.
\end{align*}
Bis jetzt haben wir erhalten:
\begin{align*}
s_1&=\sum _{k=1}^{P-1}\left(\prod_{\ell=1}^{k-1}\frac {2\,\ell-1}{2\,\ell}\right)\left(2\,k-1\right)\,\sin\left(x\right)^{2\,k-1}\;;
\\
s_2&=\sum _{k=1}^{P-1}\left(\prod_{\ell=1}^{k-1}\frac {2\,\ell-1}{2\,\ell}\right)\left(2\,k-1\right)\,\sin\left(x\right)^{2\,k-1}
\\
&+\left(2\,P-1\right)\left(\prod_{\ell=1}^{P-1}\frac {2\,\ell-1}{2\,\ell}\right)\,\sin\left(x\right)^{2\,P-1}\,.
\end{align*}
Jetzt brauchen wir nur noch einzusetzen:
\begin{align*}
S\,''\left(x\right)
=&s_1-s_2
\\
=&\sum _{k=1}^{P-1}\left(\prod_{\ell=1}^{k-1}\frac {2\,\ell-1}{2\,\ell}\right)\left(2\,k-1\right)\,\sin\left(x\right)^{2\,k-1}
\\
-&\sum _{k=1}^{P-1}\left(\prod_{\ell=1}^{k-1}\frac {2\,\ell-1}{2\,\ell}\right)\left(2\,k-1\right)\,\sin\left(x\right)^{2\,k-1}
\\
-&\left(2\,P-1\right)\left(\prod_{\ell=1}^{P-1}\frac {2\,\ell-1}{2\,\ell}\right)\,\sin\left(x\right)^{2\,P-1}\,.
\end{align*}
$\\ $
Die Summen verschwinden und für die zweite Ableitung bleibt der einfache Term:
\[
S\,''\left(x\right)=
-\left(2\,P-1\right)\left(\prod_{\ell=1}^{P-1}\frac {2\,\ell-1}{2\,\ell}\right)\,\sin\left(x\right)^{2\,P-1}\,.
\]
Setzen wir jetzt den Fixpunkt $x=\pi$ ein, ergibt das sofort $S\,''\left(\pi\right)=0\,.$
\subsubsection{Höhere Ableitungen}
Wir haben noch zu beweisen, dass für die höheren Ableitungen folgendes gilt:
\begin{align*}
S^{\,\left(k\right)}\left(\pi\right)=\left\{
\begin{aligned}
0\qquad\qquad\qquad\;\; 1\leq &k\le 2\,P
\\
\\
\left(\prod_{\ell=1}^{P}\left(2\,\ell-1\right)\right)^{2}
\qquad\quad k=&2\,P+1
\\ 
\end{aligned} 
\right.
\end{align*}
Dazu müssen wir $S\left(x\right)$ insgesamt $\left(2\,P+1\right)$ mal differenzieren
und dann $x=\pi$ einsetzen. Da die zweite Ableitung  $S\,''\left(x\right)$ bereits vorliegt, vereinfacht sich das Problem zu
\begin{align*}
\left.\bigl(S\left(x\right)\bigr)^{\,\left(2\,P+1\right)}\right|_{x=\pi}
&=\left.\bigl(S\,''\left(x\right)\bigr)^{\,\left(2\,P-1\right)}\right|_{x=\pi}
\\
&=\left.\left(-\left(2\,P-1\right)\left(\prod\limits_{l=1}^{P-1}{\frac{2l-1}{2\,\ell}}\right)\sin {{\left(x\right)}^{2\,P-1}}\right)^{\left(2\,P-1\right)}\right|_{x=\pi}
\\
&=-\left(2\,P-1\right)\left(\prod\limits_{l=1}^{P-1}{\frac{2l-1}{2\,\ell}}\right)\left.\left(\sin {{\left(x\right)}^{2\,P-1}}\right)^{\left(2\,P-1\right)}\right|_{x=\pi}\,.
\end{align*}
Die linke Seite wäre dann schon einmal richtig. Auf der rechten Seite benötigen wir im Fixpunkt $\pi$ die $\left(2\,P-1\right)$-te Ableitung der Potenz $\sin\left(x\right)^{2\,P-1}$. Dazu entwickeln wir $\sin\left(x\right)$ in eine Taylorreihe vom Grad $1$ um $\pi$ mit Lagrange-Restglied:
\[
\sin\left(x\right)
=\sin\left(\pi\right)
+\sin'\left(\pi\right)\left(x-\pi\right)
+\frac{\sin''\left(\xi\right)}{2}\left(x-\pi\right)^{2}\,,
\]
wobei $\xi$ irgendwo zwischen $x$ und $\pi$ liegt. Da $\sin\left(\pi\right)=0$ und $\sin'\left(\pi\right)=\cos\left(\pi\right)=-1$ gilt, erhalten wir
\[
\sin\left(x\right)
=-\left(x-\pi\right)+\frac{\sin''\left(\xi\right)}{2}\left(x-\pi\right)^{2}\,.
\]
Weiter können wir $\left(x-\pi\right)$ ausklammern und erhalten die einfache Darstellung:
\[
\sin\left(x\right)
=\left(x-\pi\right)\left(-1+\frac{\sin''\left(\xi\right)}{2}\left(x-\pi\right)\right)\,.
\]
Als nächstes erheben wir $\sin\left(x\right)$ in die $\left(2\,P-1\right)^{\,th}$ Potenz:
\[
\sin\left(x\right)^{2\,P-1}
=\left(x-\pi\right)^{2\,P-1}
\left(-1+\frac{\sin{\,''}\left(\xi\right)}{2}\left(x-\pi\right)\right)^{2\,P-1}\,.
\]
Die höheren Ableitungen berechnen wir mit der allgemeinen Leibniz-Regel \cite{HJ}, die besagt:
\\
Seien  $u\left(x\right)$ und $v\left(x\right)$ zwei über demselben Intervall $n$-mal differenzierbare Funktionen, dann ist das Produkt $u\left(x\right)\cdot v\left(x\right)$ ebenfalls $n$-mal differenzierbar und es gilt: 
\[
\bigl(u\left(x\right)\cdot v\left(x\right)\bigr)^{\left(n\right)}
=\sum\limits_{k=0}^{n}\binom{n}{k}\,u\left(x\right)^{\left(k\right)}\cdot v\left(x\right)^{\left(n-k\right)}
\]
Hierbei wollen wir unter $u\left(x\right)^{\left(0\right)}=u\left(x\right)$ und $v\left(x\right)^{\left(0\right)}=v\left(x\right)$ verstehen.
$\\ $
$\\ $
Angewandt auf unseren Fall bedeutet das:
\[
n=2\,P-1\;;
\]
\[
u\left(x\right)=\left(x-\pi\right)^{2\,P-1}\;;
\]
\[
v\left(x\right)=\left(-1+\frac{\sin''\left(x\right)}{2}\left(x-\pi\right)\right)^{2\,P-1}\,.
\]
Wir erhalten damit die Darstellung
\[
\sin\left(x\right)^{2\,P-1}=
\underbrace{{\left(x-\pi\right)}^{2\,P-1}}_{u\left(x\right)}
\cdot
\underbrace{{{\left(-1+\frac{\sin^{\,''}{{\left(\xi\right)}}}{2}\left(x-\pi\right)\right)}^{2\,P-1}}}_{v\left(x\right)}\,.
\]
Hierauf wenden wir die Leibniz-Regel. Das ergibt dann die Summe:
\begin{align*}
&\left(\sin\left(x\right)^{2\,P-1}\right)^{\left({2\,P-1}\right)}
\\
&=\sum\limits_{k=0}^{{2\,P-1}}\binom{{2\,P-1}}{k}
\,\left(\left(x-\pi\right)^{2\,P-1}\right)^{\left(k\right)}
\cdot
\left(\left(-1+\frac{\sin^{\,''}{{\left(\xi\right)}}}{2}\left(x-\pi\right)\right)^{{2\,P-1}}\right)^{\left({2\,P-k-1}\right)}\,.
\end{align*}
Wir sehen uns zuerst die Folge der Ableitungen an:
\[
u\left(x\right)^{\left(k\right)}
=\left(\left(x-\pi\right)^{{2\,P-1}}\right)^{\left\{k\right\}},k=0,1,\ldots,{2\,P-1}\,.
\]
Zur Berechnung der Ableitungen von
\[
\Bigl(\left(x-\pi\right)^{2\,P-1}\Bigr)^{\left(k\right)},\;k=0,1,\ldots,2\,P-1
\]
gibt es eine einfache Regel, die wir einer Formelsammlung entnehmen \cite{SO}:
\begin{align*}
\Bigl(\left(x-\pi\right)^{2\,P-1}\Bigr)^{\left(k\right)}
=\frac{\left(2\,P-1\right)\,!}{\left(2\,P-1-k\right)\,!}\left(x-\pi\right)^{2\,P-1-k}\,.
\end{align*}
$\\ $
Insbesondere stellen wir fest, dass für $k<2\,P-1$ in der Formel
\[
u\left(x\right)^{\left(k\right)}=
\left(\left(x-\pi\right)^{{2\,P-1}}\right)^{\left(k\right)}
\]
stets der Faktor $\left(x-\pi\right)$ auftritt und somit für $x=\pi$ die Terme zu Null werden. Nur im Fall $k=\left(2\,P-1\right)$ erhalten wir einen von Null verschiedenen Wert nämlich das Produkt
\[
{\left.u{{\left(x\right)}^{\left(2\,P-1\right)}}\right|}_{x=\pi}=\left(2\,P-1\right)\,!\,.
\]
Von der Summe
\[
\sum\limits_{k=0}^{{2\,P-1}}\binom{{2\,P-1}}{k}
\,\left(\left(x-\pi\right)^{2\,P-1}\right)^{\left(k\right)}
\cdot
\left(\left(-1+\frac{\sin^{\,''}{{\left(\xi\right)}}}{2}\left(x-\pi\right)\right)^{{2\,P-1}}\right)^{\left({2\,P-k-1}\right)}\,.
\]
bleibt somit für $x=\pi$ nur der Term  mit Index $k=\left(2\,P-1\right)$ übrig:
\begin{align*}
&{\left.{{\left(\sin{{\left(x\right)}^{2\,P-1}}\right)}^{\left(2\,P-1\right)}}\right|}_{x=\pi}
\\
&=
\underbrace{\sum\limits_{k={2\,P-1}}^{{2\,P-1}}\binom{{2\,P-1}}{{2\,P-1}}\left.\left(\left(x-\pi\right)^{2\,P-1}\right)^{\left(2\,P-1\right)}\right|_{x=\pi}}_{=\left(2\,P-1\right)\,!}
\cdot
\underbrace{\left.\left(\left(-1
+\frac{\sin''{\left(\xi\right)}}{2}\left(x-\pi\right)\right)^{2\,P-1}\right)\right|_{x=\pi}}_{=-1}\,.
\end{align*}
Wir erhalten das einfache Ergebnis
\[
{\left.{{\left(\sin{{\left(x\right)}^{2\,P-1}}\right)}^{\left(2\,P-1\right)}}\right|}_{x=\pi}
=-\left(2\,P-1\right)\,!\,.
\]
Unsere ursprüngliches Anliegen war die Berechnung von
\begin{align*}
\left.S^{\,\left(2\,P+1\right)}\left(x\right)\right|_{x=\pi}
=\left.\left(-\left(2\,P-1\right)\left(\prod\limits_{l=1}^{P-1}{\frac{2l-1}{2\,\ell}}\right)\sin {{\left(x\right)}^{2\,P-1}}\right)^{\left(2\,P-1\right)}\right|_{x=\pi}\,.
\end{align*}
Jetzt brauchen wir nur noch einzusetzen und erhalten:
\begin{align*}
S^{\,\left(2\,P+1\right)}\left(\pi\right)
&=-\left(2\,P-1\right)\left(\prod\limits_{\ell=1}^{P-1}{\frac{2\ell-1}{2\,\ell}}\right)
\bigl(-\left(2\,P-1\right)\,!\bigr)
\\
&=\left(2\,P-1\right)\cdot\left(2\,P-1\right)\,!\left(\prod\limits_{l=1}^{P-1}{\frac{2l-1}{2l}}\right)\,.
\end{align*}
Wir sind fast am Ziel. Wir nehmen noch eine Vereinfachung vor:
\begin{align*}
\left(2\,P-1\right)\cdot\left(2\,P-1\right)\,!\cdot
\left(\prod\limits_{\ell=1}^{P-1}{\frac{2\ell-1}{2\,\ell}}\right)
&=\frac{\left(2\,P-1\right)\,!}{\prod\limits_{\ell=1}^{P-1}\left({2\,\ell}\right)}\cdot
\left(2\,P-1\right)\cdot\prod\limits_{\ell=1}^{P-1}\left({2\ell-1}\right)
\\
&=\frac{\left(2\,P-1\right)\,!}{\prod\limits_{\ell=1}^{P-1}\left({2\,\ell}\right)}\cdot\prod\limits_{\ell=1}^{P}\left({2\ell-1}\right)\,.
\end{align*}
Als nächstes sehen wir und den Term
\[
\frac{\left(2\,P-1\right)\,!}{\prod\limits_{\ell=1}^{P-1}\left({2\,\ell}\right)}
\]
näher an. Da es sich um ein endliches Produkt handelt können wir die einzelnen Faktoren direkt hinschreiben:
\[
\frac{\left(2\,P-1\right)\,!}{\prod\limits_{\ell=1}^{P-1}\left({2\,\ell}\right)}
=\frac
{1\cdot 2\cdot 3\ldots\left({2\,P-2}\right)\left({2\,P-1}\right)}
{2\cdot 4\cdot 6\ldots\left(2\,P-2\right)\qquad\qquad}\,.
\]
Die geraden Faktoren kürzen sich heraus und es bleibt nur das Produkt der ungeraden Faktoren:
\[
\frac{\left(2\,P-1\right)\,!}{\prod\limits_{\ell=1}^{P-1}\left({2\,\ell}\right)}
=1\cdot 3\cdot 5\ldots\left({2\,P-3}\right)\left({2\,P-1}\right)
=\prod\limits_{\ell=1}^{P}\left({2\ell-1}\right)\,.
\]
Mit diesem Ergebnis gehen wir in
\[
S\left(\pi\right)^{\left(2\,P+1\right)}
=\frac{\left(2\,P-1\right)\,!}{\prod\limits_{\ell=1}^{P-1}\left({2\,\ell}\right)}\cdot\prod\limits_{\ell=1}^{P}\left({2\,\ell-1}\right)
=\left(\prod\limits_{\ell=1}^{P}\left({2\,\ell-1}\right)\right)\cdot
 \left(\prod\limits_{\ell=1}^{P}\left({2\,\ell-1}\right)\right)
\]
und erhalten schlussendlich
\begin{align*}
S^{\,\left(2\,P+1\right)}\left(\pi\right)
&=\frac{\left(2\,P-1\right)\,!}{\prod\limits_{\ell=1}^{P-1}\left({2\,\ell}\right)}\cdot\prod\limits_{\ell=1}^{P}\left({2\,\ell-1}\right)
=\left(\prod\limits_{\ell=1}^{P}\left({2\,\ell-1}\right)\right)
\cdot \left(\prod\limits_{\ell=1}^{P}\left({2\,\ell-1}\right)\right)
\\
&=\left(\prod\limits_{\ell=1}^{P}\left({2\,\ell-1}\right)\right)^2\,.
\end{align*}
Fassen wir alle Ergebnisse zusammen, so erhalten wir das gewünschte Ergebnis:
\begin{align*}
S^{\,\left(k\right)}\,{\left(\pi\right)}=\left\{
\begin{aligned}
0\qquad\qquad\qquad\;\; 1\le\,&k\leq 2\,P
\\
\\ 
\;\left(\prod\limits_{\ell=1}^{P}\left({2\ell-1}\right)\right)^2\,\;\qquad k = &2\,P+1 
\end{aligned} 
\right.
\end{align*}
Damit ist der Beweis für die Ableitungen abgeschlossen.
\subsection{Beweis der Konvergenz}
Wir führen den Beweis mit dem Fixpunktsatz von Banach und zwar in einer speziellen Form für stetig differenzierbare Abbildungen im ${\mathbb{R}}^{1}\,.$
$ \\ $
$ \\ $
\underline{Fixpunktsatz von Banach} \cite{WG}
\\
Es $U\subseteq R$ eine abgeschlossene Teilmenge; weiterhin sei $F:U\rightarrow R$ eine Abbildung mit folgenden Eigenschaften
\[
F\left(U\right)\subseteq U \quad (Selbstabbildung)
\]
Es existiere ein $0<L<1$ so dass gilt
\[
\left|F\,'\left(x\right)\right| \leq L\,,\; \text{for\,\;all} \;x\in U \quad (Kontraktionseigenschaft)
\]
Dann gilt:
\begin{itemize}
\item[(a)]
Es gibt genau einen Fixpunkt ${{x}^{*}}\in U$ von $F\,.$
\item[(b)]
Für jeden Startwert $x_0\in {U}$ konvergiert die Folge
$x_{n+1}=F\left(x_{{n}}\right)\;;\;\;n=0,1,\ldots$ gegen den Fixpunkt $x^*\,.$
\end{itemize}
$\\ $
\underline{Konvergenz-Beweis}
$\\ $
Wir haben bereits bewiesen, dass $S\left(\pi\right)=\pi$ und $S\,'\left(\pi\right)=0$ gilt. Dann existiert aus Stetigkeitsgründen eine Umgebung um den Fixpunkt $\pi$
\[U_{\delta}=\left\{x\in\mathbb{R}:\left|x-\pi\right|\leq\delta\right\}\]
in der gilt: $S\left(U_\delta\right)\subseteq U_\delta$ und $\left|S\,'\left(x\right)\right|<L<1$ für alle $x\in U_{\delta}$. Somit sind die Voraussetzungen an den Fixpunktsatz von Banach erfüllt und hieraus folgt die Konvergenz.
\subsection{Beweis der Konvergenzordnung}
Zuletzt haben wir noch zu beweisen, dass die Konvergenzordnung genau $\left(2\,P+1\right)$ beträgt. Wir greifen auf bewährtes zurück: Taylorreihen. Diesmal entwickeln wir $S\left(x\right)$ in eine Taylorreihe um den Fixpunkt $\pi:$
\[
S\left(x\right)=
S\left(\pi\right)
+S\,'\left(\pi\right)\left(x-\pi\right)
+\ldots
+S\,^{\,\left(2\,P\right)}\left(\pi\right)
\frac{\left(x-\pi\right)^{2\,P}}{\left(2\,P\right)\,!}
+S\,^{\,\left(2\,P+1\right)}\left(\xi\right)
\frac{\left(x-\pi\right)^{2\,P+1}}{\left(2\,P+1\right)\,!}\,,
\]
wobei $\xi$ irgendwo zwischen $x$ and $\pi$ liegt. Wir erinnern uns an die Fixpunkteigenschaft $S\left(\pi\right)=\pi$ and die Ableitungsergebnisse:
\begin{align*}
S^{\,\left(k\right)}\left(\pi\right)=\left\{
\begin{aligned}
0\qquad\qquad\qquad\;\; 1\leq &k\le 2\,P
\\
\\
\left(\prod_{\ell=1}^{P}\left(2\,\ell-1\right)\right)^{2}
\qquad\quad k=\,&2\,P+1
\\ 
\end{aligned} 
\right.
\end{align*}
Die Taylorreihe vereinfacht sich dadurch zu
\[
S\left(x\right)=
\underbrace{S\left(\pi\right)}_{=\pi}
+\underbrace{S\,'\left(\pi\right)\left(x-\pi\right)
+\ldots
+S\,^{\,\left(2\,P\right)}\left(\pi\right)
\frac{\left(x-\pi\right)^{2\,P}}{\left(2\,P\right)\,!}}_{=\,0}
+S\,^{\,\left(2\,P+1\right)}\left(\xi\right)
\frac{\left(x-\pi\right)^{2\,P+1}}{\left(2\,P+1\right)\,!}
\]
\[
=
\pi
+S\,^{\,\left(2\,P+1\right)}\left(\xi\right)
\frac{\left(x-\pi\right)^{2\,P+1}}{\left(2\,P+1\right)\,!}\,.
\]
Jetzt setzen wir $x=x_n$ hinreichend nahe bei $\pi$ ein und erhalten
\begin{align*}
S\left(x_{n}\right)=\pi+S\,^{\left(2\,P+1\right)}\left(\xi_{n}\right)
\frac{\left(x_{n}-\pi\right)^{2\,P+1}}{\left(2\,P+1\right)\,!}\,,
\end{align*}
mit einem $\xi_n$ zwischen $x_n$ und $\pi$. Da $S\left(x_n\right)=x_{n+1}$, ergibt sich
\begin{align*}
x_{n+1}=\pi
+S\,^{\left(2\,P+1\right)}\left(\xi_{n}\right)
\frac{\left(x_{n}-\pi\right)^{2\,P+1}}{\left(2\,P+1\right)\,!}\,.
\end{align*}
Als nächstes bringen wir $\pi$ auf die linke Seite und dividieren durch $\left(x_{n+1}-\pi\right)^{2\,P+1}$. Wir erhalten damit:
\[
\frac{x_{n+1}-\pi}{\left(x_{n}-\pi\right)^{2\,P+1}}=
\frac{S\,^{\left(2\,P+1\right)}\left(\xi_{n}\right)}{\left(2\,P+1\right)\,!}\,.
\]
Mit zunehmendem $n$ strebt $x_n\rightarrow \pi$ mit Konvergenzordnung $\left(2\,P+1\right)$. Da $\xi_n$ zwischen $x_n$ und $\pi$ eingesperrt ist, muss es notwendigerweise ebenfalls gegen $\pi$ streben. Wir erhalten:
\begin{align*}
\underset{n\to \infty}{\mathop{\lim }}\; \frac{x_{n+1}-\pi}{\left(x_{n}-\pi\right)^{2\,P+1}}
=\underset{n\to \infty}{\mathop{\lim }}\;
\frac{S\,^{\left(2\,P+1\right)}\left(\xi_{n}\right)}{\left(2\,P+1\right)\,!}
=
\frac{S\,^{\left(2\,P+1\right)}\left(\pi\right)}{\left(2\,P+1\right)\,!}\,.
\end{align*}
Nun haben wir bereits berechnet:
\[
S\,{^{\left({2\,P+1}\right)}\left(\pi\right)}
=\left(\prod\limits_{\ell=1}^{P}\left({2\,\ell-1}\right)\right)^2\,.
\]
Verwenden wir dies, so erhalten wir
\begin{align*}
\underset{n\to \infty}{\mathop{\lim }}\;\;
\frac{x_{n+1}-\pi}{\left(x_{n}-\pi\right)^{2\,P+1}}
=\frac{1}{\left(2\,P+1\right)\,!}\left(\prod\limits_{\ell=1}^{P}\left({2\,\ell-1}\right)\right)^2\,.
\end{align*}
Dies zeigt, dass die Konvergenzordnung genau $\left(2\,P+1\right)$ beträgt. Der auf der rechten Seite stehende Term
\[
\frac{1}{\left(2\,P+1\right)\,!}\left(\prod\limits_{\ell=1}^{P}\left({2\,\ell-1}\right)\right)^2
\]
wird als asymtotische Konvergenzrate bezeichnet.
\section{Praktische Vorgehensweise}
Wir geben $P\in\mathbb{N}$ vor und erhalten die Konvergenzordnung $\left(2\,P+1\right)$. Als nächstes bestimmen wir die Fixpunktfunktion
\[
S\left(x\right)=x+\sum _{k=1}^{P}\left(\left(\prod_{\ell=1}^{k-1}\frac {2\,\ell-1}{2\,\ell}\right)\frac{\sin\left(x\right)^{2\,k-1}}{2\,k-1}\right)\,.
\]
Als erstes stellen wir fest, dass sich die etwas kompliziert gebauten Koeffizienten während der gesamten Berechnung nicht ändern. Es ist deshalb zweckmäßig diese vorab zu bestimmen, und die einzelnen Elemente in einem Koeffizienten-Vektor abzuspeichern:
\[
c{\,_k}=\left(\prod_{\ell=1}^{k-1}\frac {2\,\ell-1}{2\,\ell}\right)\cdot\frac{1}{2\,k-1}\;;k=1\,,\ldots,\,{P}\,.
\]
Die Fixpunktfunktion vereinfacht sich dadurch zu
\[S\left(x\right)=\sum\limits_{k=1}^{P}c_{\,k}\,\sin\left(x\right)^{2\,k-1}.\]
Hiervon leiten wir die Iterations-Vorschrift ab:
\[x_{n+1}=x_n+\sum\limits_{k=1}^{P}c_{\,k}\,\sin\left(x_n\right)^{2\,k-1}\;;\;n=0,1,\ldots\]
Bei der praktischen Berechnung ist eine Indexverschiebung sinnvoll. Statt
\[x_{n+1}=x_n+\sum\limits_{k=1}^{P}c_{\,k}\,\sin\left(x_n\right)^{2\,k-1}\;;\;n=0,1,\ldots\]
schreiben wir
\[x_{n}=x_{n-1}+\sum\limits_{k=1}^{P}c_{\,k}\,\sin\left(x_{n-1}\right)^{2\,k-1}\;;\;n=1,2,\ldots\,.\]
Damit läuft der Index $n$ synchron mit der Anzahl der Iterationsschritte. Mit einem geeigneten Startwert $x_0$ berechnen wir die Folge:
\[\underline{step\;\;1}\]
\[x_{1}=x_{0}+\sum\limits_{k=1}^{P}c_{\,k}\,\sin\left(x_{0}\right)^{2\,k-1}\,;\]

\[\underline{step\;\;2}\]
\[x_{2}=x_{1}+\sum\limits_{k=1}^{P}c_{\,k}\,\sin\left(x_{1}\right)^{2\,k-1}\,;\]

\[\vdots\]
\[\underline{step\;\;n}\]
\[x_{n}=x_{n-1}+\sum\limits_{k=1}^{P}c_{\,k}\,\sin\left(x_{n-1}\right)^{2\,k-1}\,;\]
$ \\ $
Mit zunehmender Anzahl $n$ der Iterationsschritte strebt $x_n\rightarrow\pi$ mit Konvergenzordnung $\left(2\,P+1\right)$. Bei beliebig langer Rechenzeit und beliebig viel Speicherplatz können wir damit $\pi$ beliebig genau berechnen. Leider steht uns nichts dergleichen zur Verfügung. Wir müssen uns mit einer endlichen Anzahl von Iterationen und einem endlichen Wert für $\pi$ begnügen.
$ \\ $
Dazu geben wir $\epsilon>0$ vor und führen die Iteration solange aus bis erstmals gilt:
\[\left|x_n-x_{n-1}\right|<\epsilon\,.\]
\subsection{\texorpdfstring{Erste Berechnung von\;$\pi$}{}}
All Abbildungen und numerischen Berechnungen wurden mit Hilfe des Computer-Algebra-Systems \textbf{Maple 2025.2} erstellt und durchgeführt.
$ \\ $
$ \\ $
In der ersten $\pi$-Berechnung sei $P=1$. Das ergibt die Konvergenzordnung $\left(2\,P+1\right)=3$. Das heißt, die Anzahl der gültigen Stellen verdreifacht sich näherungsweise mit jedem weiterem Iterations-Schritt. Wir haben nur einen einzigen Koeffizienten $c_{\,1}=1\,.$ Die Fixpunktfunktion ist somit sehr einfach:
\[S\left(x\right)=x+\sin\left(x\right)\,.\]
Wir sehen uns die beteiligten Funktionen $\sin\left(x\right)$ und $S\left(x\right)$ zusammen mit der Geraden $y=x$ einmal in einer Abbildung an:
\begin{center}
\includegraphics[width=1.0\linewidth]{"fp_sin_P1.eps"}
\end{center}
\textbf{\underline{figure 1\,:}}$\;\;\sin\left(x\right),\;S\left(x\right)=x+\sin\left(x\right)$ and $y=x$
$\\ $
$\\ $
Aus $S\left(x\right)$ leiten wir die Iteration ab: \[x_{n}=x_{n-1}+\sin\left(x_{n-1}\right),.\]
Wir starten mit $x_0=3$ berechnen sieben Iterationschritte
\[\underline{Step\;\;1}\]
\[x_{{1}}=x_{{0}}+\sin\left(x_{{0}}\right)\]
\[x_{{1}}=3.141120008059867222100744802808110279847\]
\[\left|x_{{1}}-x_{{0}}\right|=0.1411200080598672221007448028081102798469\]
\[\underline{Step\;\;2}\]
\[x_{{2}}=x_{{1}}+\sin\left(x_{{1}}\right)\]
\[x_{{2}}=3.141592653572195558734888568140879746743\]
\[\left|x_{{2}}-x_{{1}}\right|=0.0004726455123283366341437653327694668960596\]
\[\underline{Step\;\;3}\]
\[x_{{3}}=x_{{2}}+\sin\left(x_{{2}}\right)\]
\[x_{{3}}=3.141592653589793238462643383279501975927\]
\[\left|x_{{3}}-x_{{2}}\right|={1.759767972775481513862222918415964000597\times10^{-11}}\]

\[\underline{Step\;\;4}\]
\[x_{{4}}=x_{{3}}+\sin\left(x_{{3}}\right)\]
\[x_{{4}}=3.141592653589793238462643383279502884197\]
\[\left|x_{{4}}-x_{{3}}\right|={9.082700169421541422061207647042849447509\times10^{-34}}\]

\[\underline{Step\;\;5}\]
\[x_{{5}}=x_{{4}}+\sin\left(x_{{4}}\right)\]
\[x_{{5}}=3.141592653589793238462643383279502884197\]
\[\left|x_{{5}}-x_{{4}}\right|={1.248802280614662060876298822692655898920\times10^{-100}}\]

\[\underline{Step\;\;6}\]
\[x_{{6}}=x_{{5}}+\sin\left(x_{{5}}\right)\]
\[x_{{6}}=3.141592653589793238462643383279502884197\]
\[\left|x_{{6}}-x_{{5}}\right|={3.245860113595057511363635505332694019725\times10^{-301}}\]

\[\underline{Step\;\;7}\]
\[x_{{7}}=x_{{6}}+\sin\left(x_{{6}}\right)\]
\[x_{{7}}=3.141592653589793238462643383279502884197\]
\[x_{{7}}-x_{{6}}={5.699518230086813078688904262462295801452\times10^{-903}}\]
$\\ $
Nach sieben Iterationen haben wir bereits einen auf über 900 Stellen genauen Wert für $\pi$ erhalten. An den negativen Exponenten der Differenzen können wir sehr schön die Konvergenzordnung $3$ ablesen. Der negative Exponent verdreifacht sich näherungsweise mit jedem weiteren Schritt.
\subsection{\texorpdfstring{Berechnung einer Million Stellen von\;$\pi$}{}}
Das nächste Beispiel ist etwas anspruchsvoller. Diesmal wollen wir $\pi$ auf mindestens eine Million Dezimalstellen genau berechnen. Das bedeutet $\epsilon=10^{-1000000}$.
Um die Fixpunktfunktion zur Berechnung von $\pi$ verwenden zu können, ist eine effiziente Berechnung von $\sin\left(x\right)$ erforderlich. Es gibt jede Menge Veröffentlichungen zu diesem Thema \cite{RB},\cite{FJ1},\cite{FJ2},\cite{HG},\cite{RBA}. Aus der Vielzahl der vorgeschlagenen Verfahren haben wir die folgenden drei ausgewählt:
$\\ $
$\\ $
\underline{\bf{1.\;Verwendung eines Startwertes, der schon sehr nahe bei $\pi$ liegt}}
$\\ $
We verwenden den Startwert \[x_0=3.14159265358979324.\] Dieser ist bereits auf 18 Dezimalstellen genau.
$ \\ $
$ \\ $
\underline{\bf{2.\;Erhöhung der Konvergenzordnung}}
$\\ $
Um sicher zustellen, dass die Konvergenzordnung voll zum Tragen kommt, wählen wir $P=4$. Das ergibt die Konvergenzordnung $2\cdot 4+1=9$. Das bedeutet, dass sich die Anzahl der gültigen Dezimalstellen mit jedem Schritt verneunfacht. Wir benötigen insgesamt 6 Iterationsschritte um eine Genauigkeit auf mindestens einer Million Stellen für $\pi$ zu erhalten.
$ \\ $
$ \\ $
\underline{\bf{3.\;Schritt haltend die Genauigkeit erhöhen}}
$\\ $
Die Fixpunktiteration
\[x_{n}=x_{n-1}+\sum\limits_{k=1}^{P}c_{\,k}\,\sin\left(x_{n-1}\right)^{2\,k-1}\]
ist selbst-korrigierend. Wir müssen nicht von Anfang an mit der Genauigkeit von einer Million Stellen rechnen. Wir können mit weitaus geringerer Genauigkeit loslegen.
$\\ $
$\\ $
Die Variable \textbf{Digits} gibt die Anzahl der Dezimalstellen an mit denen Maple floating-point Zahlen berechnet; gibt man zum Beispiel \textbf{Digits}=100 vor, so rechnet Maple mit 100 Stellen Genauigkeit. Im ersten Schritt benötigen wir eine Genauigkeit von $18\cdot 9=162$ Stellen, also ist \textbf{Digits}=162. Im 2. Schritt liegt die benötigte Genauigkeit bei $18\cdot 9^2=1458$ Stellen, also müssen wir \textbf{Digits}=1458 setzen. So geht es weiter bis zum fünften Schritt, wo Maple mit $18\cdot 9^5=1062882$ Dezimalstellen rechnet. Jetzt haben wir die vorgegebene Anzahl von mindestens einer Million Stellen für $\pi$ erreicht. Deshalb brauchen wir im letzten Schritt keine weitere Erhöhung der Stellenanzahl vornehmen.
$ \\ $
$ \\ $
Zuerst bestimmen wir die Koeffizienten
\[c_1=1\;;\;c_2=\frac{1}{6}\;;\;c_3=\frac{3}{40}\;;\;c_4=\frac{5}{112}\]
und die Fixpunktfunktion:
\[
S\left(x\right)
=x+\sin\left(x\right)
+\frac{1}{6}\,\sin\left(x\right)^{3}
+\frac{3}{40}\,\sin\left(x\right)^{5}
+\frac{5}{112}\,\sin\left(x\right)^{7}\,.
\]
Hiervon leiten wir die Iterationsformel ab:
\[x_{n}=x_{n-1}+\sin\left(x_{n-1}\right)
+\frac{1}{6}\sin\left(x_{n-1}\right)^{3}
+\frac{3}{40}\,\sin\left(x_{n-1}\right)^{5}
+\frac{5}{112}\,\sin\left(x_{n-1}\right)^{7}\;.\]
$ \\ $
Bevor wir mit der Berechnung beginnen sehen wir uns die beteiligten Funktionen in einer weiteren Abbildung an:
\begin{center}
\includegraphics[width=0.85\linewidth]{"fp_sin_P4.eps"}
\end{center}
%\color{black}
%\color{Sepia}
\textbf{\underline{figure 2\,:\;\;}}
$\sin\left(x\right),\;$ $S\left(x\right)
=x+\sin\left(x\right)
+\frac{1}{6}\,\sin\left(x\right)^{3}
+\frac{3}{40}\,\sin\left(x\right)^{5}
+\frac{5}{112}\,\sin\left(x\right)^{7}$ und $y=x$
$ \\ $
$ \\ $
In der nächsten $\pi$ Berechnung haben wir die drei oben beschriebenen Optimierungsverfahren eingebaut. Wir starten mit \[x_0=3.14159265358979324\] und berechnen 6 Iterationen. Zusätzlich geben wir für jeden Schritt die Anzahl der \textbf{Digits} an, mit denen Maple gerade rechnet.

\[\underline{step\;\;1}\]
\[Digits=18\cdot 9=162\]
\[x_{{1}}=x_{{0}}+\sin\left(x_{{0}}\right)+{\frac{\sin\left(x_{{0}}\right)^{3}}{6}}+{\frac{3\,\sin\left(x_{{0}}\right)^{5}}{40}}+{\frac{5\,\sin\left(x_{{0}}\right)^{7}}{112}}\]
\[x_{{1}}=3.141592653589793238462643383279502884197\]
\[\left|x_{{1}}-x_{{0}}\right|={1.537356616720497115802830600624894179025\times10^{-18}}\]

\[\underline{step\;\;2}\]
\[Digits=18\cdot 9^2=1458\]
\[x_{{2}}=x_{{1}}+\sin\left(x_{{1}}\right)+{\frac{\sin\left(x_{{1}}\right)^{3}}{6}}+{\frac{3\,\sin\left(x_{{1}}\right)^{5}}{40}}+{\frac{5\,\sin\left(x_{{1}}\right)^{7}}{112}}\]
\[x_{{2}}=3.141592653589793238462643383279502884197\]
\[\left|x_{{2}}-x_{{1}}\right|={5.897298061478894440355377051045069618036\times10^{-162}}\]
\[\underline{step\;\;3}\]
\[Digits=18\cdot 9^3=13122\]
\[x_{{3}}=x_{{2}}+\sin\left(x_{{2}}\right)+{\frac{\sin\left(x_{{2}}\right)^{3}}{6}}+{\frac{3\,\sin\left(x_{{2}}\right)^{5}}{40}}+{\frac{5\,\sin\left(x_{{2}}\right)^{7}}{112}}\]
\[x_{{3}}=3.141592653589793238462643383279502884197\]
\[\left|x_{{3}}-x_{{2}}\right|={2.621201614945054114130730043090727892025\times10^{-1453}}\]
\[\underline{step\;\;4}\]
\[Digits=18\cdot 9^4=118098\]
\[x_{{4}}=x_{{3}}+\sin\left(x_{{3}}\right)+{\frac{\sin\left(x_{{3}}\right)^{3}}{6}}+{\frac{3\,\sin\left(x_{{3}}\right)^{5}}{40}}+{\frac{5\,\sin\left(x_{{3}}\right)^{7}}{112}}\]
\[x_{{4}}=3.141592653589793238462643383279502884197\]
\[\left|x_{{4}}-x_{{3}}\right|={1.774677351983440263736354294020071931434\times10^{-13075}}\]
\[\underline{step\;\;5}\]
\[Digits=18\cdot 9^5=1062882\]
\[x_{{5}}=x_{{4}}+\sin\left(x_{{4}}\right)+{\frac{\sin\left(x_{{4}}\right)^{3}}{6}}+{\frac{3\,\sin\left(x_{{4}}\right)^{5}}{40}}+{\frac{5\,\sin\left(x_{{4}}\right)^{7}}{112}}\]
\[x_{{5}}=3.141592653589793238462643383279502884197\]
\[\left|x_{{5}}-x_{{4}}\right|={5.305059119433192986813252324813505545638\times10^{-117675}}\]
\[\underline{step\;\;6}\]
\[Digits=1062882\]
\[x_{{6}}=x_{{5}}+\sin\left(x_{{5}}\right)+{\frac{\sin\left(x_{{5}}\right)^{3}}{6}}+{\frac{3\,\sin\left(x_{{5}}\right)^{5}}{40}}+{\frac{5\,\sin\left(x_{{5}}\right)^{7}}{112}}\]
\[x_{{6}}=3.141592653589793238462643383279502884197\]
\[\left|x_{{6}}-x_{{5}}\right|={1.011178012461738293576866860651780969327\times10^{-1059070}}\]
$\\ $
Nach sechs Iterationsschritten haben wir weit über eine  Million Dezimalstellen für $\pi$ erhalten. Genau genommen ergab sich für $\pi$ eine Genauigkeit in der Größenordnung von 1059070 Dezimalstellen. An den negativen Potenzen der Differenzen können wir die Konvergenzordnung $9$ ablesen. Sie nehmen näherungsweise um den Faktor 9 zu.
$\\ $
$\\ $
Alle Berechnungen wurden auf einem Eigenbau-PC mit folgender Hardware ausgeführt: Motherboard ASUS PRIME A320M-K mit CPU AMD Ryzen 5 5600G 6 CORE 3.90-4.40 GHz und 32 GB RAM. Die verwendete Software war MAPLE 2025.2 von  Maplesoft, Waterloo Maple Inc. Die Software wurde ebenfalls für die Konvertierung der mathematischen Terme nach LaTeX verwendet.

\end{document}